\documentclass[11pt,a4paper]{article}

\usepackage[utf8]{inputenc}

\usepackage{array,amsbsy,amscd,amsfonts,color,amssymb,amstext,amsmath,latexsym,amsthm,amscd
,multicol,graphicx,makeidx,tikz}
\usepackage[english]{babel}
\usepackage{hyperref}

\numberwithin{equation}{section}

\makeindex

\newtheorem{theorem}{Theorem}
\newtheorem{proposition}[theorem]{Proposition}

\newtheorem{lemma}[theorem]{Lemma}
\newtheorem{corollary}[theorem]{Corollary}

\numberwithin{theorem}{section}

\theoremstyle{definition}
\theoremstyle{definition}\newtheorem{definition}[theorem]{Definition}
\theoremstyle{definition}\newtheorem{remark}[theorem]{Remark}
\theoremstyle{definition}
\theoremstyle{definition}
\theoremstyle{definition}
\theoremstyle{definition}

\def\proofof [#1] {\noindent {\bf Proof of #1. } }

\def\al #1.{{\mathcal{#1}}}

\renewcommand{\AA}{\mathfrak{A}}

\newcommand{\A}{\mathcal{A}}
\newcommand{\B}{\mathcal{B}}
\newcommand{\mC}{\mathcal{C}}
\newcommand{\D}{\mathcal{D}}
\newcommand{\I}{\mathcal{I}}
\newcommand{\K}{\mathcal{K}}
\renewcommand{\H}{\mathcal{H}}

\newcommand{\N}{\mathbb{N}}
\newcommand{\NN}{\mathbb{N}_0}
\newcommand{\R}{\mathbb{R}}
\newcommand{\C}{\mathbb{C}}

\renewcommand{\S}{{S^1}}

\newcommand{\Z}{\mathbb{Z}}

\newcommand{\unit}{\mathbf{1}}

\newcommand{\bp}{\begin{proof}}
\newcommand{\ep}{\end{proof}}
\newcommand{\bdp}{\begin{dproof}}
\newcommand{\edp}{\end{dproof}}
\newcommand{\ra}{\rightarrow}

\newcommand{\Uone}{\operatorname{U}(1)}

\newcommand{\SUtwo}{\operatorname{SU}(2)}

\newcommand{\Diff}{\operatorname{Diff}(\S)}

\newcommand{\CAR}{\operatorname{CAR}}
\newcommand{\PSL}{\operatorname{PSL(2,\R)}}

\newcommand{\PSI}{\operatorname{PSL(2,\R)}^{(\infty)}}

\newcommand{\SVir}{\operatorname{SVir}}

\newcommand{\tr}{\operatorname{tr }}
\newcommand{\Ad}{\operatorname{Ad }}

\newcommand{\const}{\operatorname{const }}
\newcommand{\rmd}{\operatorname{d}}
\newcommand{\rmi}{\operatorname{i}}

\newcommand{\rme}{\operatorname{e}}

\newcommand{\de}{\delta}

\newcommand{\sgn}{\operatorname{sgn }}

\newcommand{\supp}{\operatorname{supp }}

\newcommand{\dom}{\operatorname{dom }}

\renewcommand{\ker}{\operatorname{ker }}

\newcommand{\Cci}{C^{\infty}}

\newcommand{\boxy}{\hfill $\Box$\vspace{5mm}}

\newcommand{\ie}{{i.e.,\/}\ }

\newcommand{\eg}{{e.g.\/}\ }
\newcommand{\cf}{{cf.\/}\ }

\title{
{\Huge $N=2$ Superconformal Nets}\footnotetext{Work supported in part by the ERC Advanced Grant 227458  OACFT ``Operator Algebras and Conformal Field Theory"}
}

\author{
{\sc Sebastiano Carpi}\footnote{Supported in part by PRIN-MIUR and GNAMPA-INDAM}
\\
{\small Dipartimento di Economia, Universit\`a di Chieti-Pescara ``G. d'Annunzio''}\\
{\small Viale Pindaro, 42, I-65127 Pescara, Italy} \\
{\small E-mail: {\tt s.carpi@unich.it}}
\\[0,3cm]
{\sc Robin Hillier}
\\
{\small Department of Mathematics and Statistics,
Lancaster University}\\
{\small  Lancaster LA1 4YF, UK}\\
{\small E-mail: {\tt r.hillier@lancaster.ac.uk}}
\\[0,3cm]
{\sc Yasuyuki Kawahigashi}\footnote{Supported in part by Global COE Program ``The research and training center for new development in mathematics'', the Mitsubishi Foundation Research Grants and the Grants-in-Aid for Scientific Research, JSPS.}\\
{\small Department of Mathematical Sciences}\\
{\small The University of Tokyo, Komaba, Tokyo, 153-8914, Japan}
\\[0,05cm]
{\small and}
\\[0,05cm]
{\small Kavli IPMU (WPI), the University of Tokyo}\\
{\small 5-1-5 Kashiwanoha, Kashiwa, 277-8583, Japan}\\
{\small E-mail: {\tt yasuyuki@ms.u-tokyo.ac.jp}}
\\[0,3cm]
{\sc Roberto Longo}$^{*}$
\\
{\small Dipartimento di Matematica,
Universit\`a di Roma ``Tor Vergata''}\\
{\small Via della Ricerca Scientifica, 1, I-00133 Roma, Italy}\\
{\small E-mail: {\tt longo@mat.uniroma2.it}}
\\[0,3cm]
{\sc Feng Xu}\footnote{Supported in part by NSF grant and an academic senate grant from UCR.}\\
{\small Department of Mathematics}\\
{\small University of California at Riverside, Riverside, CA 92521}\\
{\small E-mail: {\tt xufeng@math.ucr.edu}}
}

\date{}

\begin{document}
\maketitle

\begin{abstract}
We provide an Operator Algebraic approach to $N=2$ chiral Conformal Field Theory and set up the Noncommutative Geometric framework. Compared to the $N=1$ case, the structure here is much richer. There are naturally associated nets of spectral triples and the JLO cocycles separate the Ramond sectors. We construct the $N=2$ superconformal nets of von Neumann algebras in general, classify them in the discrete series $c<3$, and we define and study an operator algebraic version of the $N=2$ spectral flow. We prove the coset identification for the $N=2$ super-Virasoro nets with $c<3$, a key result whose equivalent in the vertex algebra context has seemingly not been completely proved so far. Finally, the chiral ring is discussed in terms of net representations.
\end{abstract}

\tableofcontents

\section{Introduction}
Quantum Field Theory (QFT) describes a quantum system with infinitely many degrees of freedom and, from a geometrical viewpoint, can be regarded as an infinite-dimensional noncommutative manifold. It thus becomes a natural place for merging the classical infinite-dimensional calculus with the noncommutative quantum calculus. As explained in \cite{L(QFT index)}, a QFT index theorem should manifest itself in this setting and noncommutative geometry should provide a natural framework.

Within this program, localized representations with finite Jones index should play a role analogous to the one of elliptic operators in the classical framework. One example of this structure was suggested in the black hole entropy context, the Hamiltonian was regarded in analogy with the (infinite dimensional promotion of the) Laplacian, and spectral analysis coefficients were indeed identified with index invariants for the net and its representations \cite{KLbh}. 
According to Connes \cite{Co1}, the notion of spectral triple abstracts and generalizes the notion of Dirac operator on a spin manifold to the noncommutative context and this naturally leads to exploring the supersymmetric context where the supercharge operator plays the role of the Dirac operator.

A particularly interesting context where to look for this setting is provided by chiral Conformal Field Theory in two spacetime dimensions (CFT), a building block for general 2D CFT. There are several reasons why CFT is suitable for our purposes. On the one hand, the Operator Algebraic approach to QFT has been particularly successful within the CFT frame leading to a deep, model independent description and understanding of the underlying structure. On the other hand, there are different, geometrically based approaches to CFT suggesting a noncommutative geometric interpretation ought to exist, and in which fields represent the noncommutative variables,
see \eg \cite{FroGad}. Since the root of Connes' noncommutative geometry is operator algebraic, one is naturally led to explore its appearance within local conformal nets of von Neumann algebras. In particular, this approach connects subfactor theory and noncommutative geometry.

A first step in this direction was taken in \cite{CKL} with the construction and structure analysis of the $N=1$ superconformal nets of von Neumann algebras, the prime class of nets combining conformal invariance and supersymmetry. Indeed according to the above comments,
the natural QFT models where spectral triples are found are the supersymmetric ones where the supercharge operator is an odd square root of the Hamiltonian. 

Indeed, nets of spectral triples have been later constructed in \cite{CHKL}, associated with Ramond representations of  the $N=1$ super-Virasoro net, the most elementary superconformal net of von Neumann algebras. In particular the irreducible, unitary positive energy representation of the Ramond $N=1$ super Virasoro algebra with central charge $c$ and minimal lowest weight $h = c/24$ is graded and gives rise to a net of even $\theta$-summable spectral triples with non-zero Fredholm index. 
More recently, three of us started in \cite{CHL} a more systematic analysis of the noncommutative geometric aspects of the superselection structure of ($N=1$) superconformal nets. In particular, they defined spectral triples and corresponding entire cyclic cocycles  associated to representations of the underlying nets in various relevant $N=1$ superconformal field theory models and proved that the cohomology classes of these cocycles encode relevant information about the corresponding sectors. On the other hand, a related K-theoretical analysis of the representation theory of conformal nets has been initiated in \cite{CCHW,CCH}.

The $N=1$ super-Virasoro algebras  (Neveu-Schwarz or Ramond) are infinite-dimensional Lie superalgebras generated by the Virasoro algebra and the Fourier modes of one Fermi field of conformal dimension $3/2$. There are higher level super-Virasoro algebras: the $N=2$ ones are generated by the Virasoro algebra and the Fourier modes of two Fermi fields and a $U(1)$-current that generates rotations associated with the symmetry of the two Fermi fields.  $N=2$ superconformal nets will be extensions of a net associated with the $N=2$ super-Virasoro nets.
One may continue the procedure even to $N=4$, where four Fermi fields are present, acted upon by $SU(2)$-currents \cite{KR-N4}. The various supersymmetries play crucial roles in several physical contexts, in particular in phase transitions of solid state physics and on the worldsheets of string theory.

This paper is devoted to the construction and analysis of the $N=2$ superconformal nets. As is known, the passage from the $N=1$ to the $N=2$ case is not a matter of generalizing and extending results because a new and more interesting structure does appear by considering $N=2$ superconformal models, although the definition of the respective nets is similar.

After summing up basic general preliminaries, we begin our analysis in Sect. \ref{sec:SVir-net}, of course, constructing the $N=2$ super-Virasoro nets of von Neumann algebras by ``integrating'' the corresponding infinite-dimensional Lie superalgebra unitary (vacuum) representations and proving the necessary local energy bounds. 
For any given value of the central charge $c$ corresponding to some unitary representation the $N=2$ super-Virasoro algebras we can define the corresponding super-Virasoro net with central charge $c$. Different values of $c$ give rise to nonisomorphic nets.
The representations of the net will correspond to the representations of the $N=2$ super-Virasoro algebras (Neveu-Schwarz or Ramond)\cite{BFK,DVYZ}, where, in the Ramond case, the representations are actually solitonic. This goes all in complete analogy to the $N=1$ case \cite{CKL}.

At this point, however, there appears a remarkable new feature of the $N=2$ super-Virasoro algebra: the appearance of the spectral flow,  a ``homotopy'' equivalence between the Neveu-Schwarz and the Ramond algebra in the sense that there exists a deformation of one into the other. 
In Sect. \ref{sec:spfl} we set up an operator algebraic version of the $N=2$ spectral flow. We find that for any value of the flow parameter it gives rise to covariant solitons of the $N=2$ super-Virasoro nets. Moreover, it has a natural interpretation in terms of $\alpha$-induction 
\cite{BE,LR}. As a consequence solitonic Ramond representations of the nets are thus in correspondence with true (DHR) representations, an important fact of later use to us.

Before proceeding further, Sect. \ref{sec:coset} is devoted to clarifying a key point of our paper: the  identification for the 
even (Bose) subnet of the
$N=2$ super-Virasoro nets with $c<3$ as a coset for the inclusion 
$\A_{\Uone_{2n+4}} \subset \A_{\SUtwo_n} \otimes \A_{\Uone_{ 4}}$. This identification is equivalent to the corresponding coset identification at the Lie algebra (or vertex algebra) level and it is moreover equivalent to the identification of the corresponding characters, cf. \cite{KL1,CKL} for the analogous statements in the $N=0,1$ cases. Accordingly, it is equivalent to the correctness of the known $N=2$ character formulae for the discrete series representations, see e.g. \cite{EG,Do98}. 

The $N=2$ character formulae for the unitary representations with $c<3$ were first derived (independently) by Dobrev \cite{Dobrev87}, 
Kiritsis \cite{Kiritsis} and Matsuo \cite{Matsuo}. Although these formulae appear to be universally accepted, a closer look at the literature 
seems to indicate that the mathematical validity of the proofs which have been proposed so far and of related issues of the representation theory of the $N=2$ superconformal algebras is rather controversial, see \cite{Do96,EG,Do98,Dobrev07,Gr} and 
\cite{FSST,FST,Gato-Rivera99,Gato-Rivera08}. For this reason, we think that it is useful to give in this paper an independent complete mathematical proof of the $N=2$ coset identification (and consequently of the $N=2$ character formulae). Our proof, which we believe in any case to be of independent interest, is obtained largely through operator algebraic methods, a point that is certainly emblematic of the effectiveness and power of operator algebras. 
Surprisingly, our operator algebraic version of the $N=2$ spectral flow plays a crucial role in the proof. As a consequence of the coset identification the Bose subnets of the super-Virasoro nets with $c<3$ turn out to be completely rational in the sense of \cite{KLM}, and the fusion rules of the corresponding sectors agree with the CFT ones. In particular, the irreducible representations of these nets give rise to finite index subfactors.

We can then proceed with the classification of the $N=2$ superconformal (chiral) minimal models in Sect. \ref{sec:classification}, \ie  
the irreducible graded-local extensions of the $N=2$ super-Virasoro net. As in the local ($N=0$) case \cite{KL1}, there are simple current series (simple current extensions, \ie to crossed products by cyclic groups) and exceptional nets (mirror extensions \cite{X-m}). The proof is again based on combinatorics, modular invariants, and subfactor methods. Compared to the cases $N=0$ \cite{KL1} and $N=1$
\cite{CKL}, however, a new phenomenon appears, namely, simple current extensions with cyclic groups of arbitrary finite order.

The noncommutative geometric analysis starts in Sect. \ref{sec:STs}, where we construct the nets of spectral triples associated with representations of the Ramond $N=2$ super-Virasoro algebra. 

The main results in noncommutative geometry then are collected in Sect. \ref{sec:JLO}, where we consider the JLO cocycles for suitable ``global'' spectral triples and pair them with K-theory. 
This pairing is nondegenerate and allows us to separate, by means of certain characteristic projections, all Ramond sectors of the $N=2$ super-Virasoro nets. An essential point here is that, in contrast with the $N=1$ case, all the irreducible representations of the Ramond $N=2$ super-Virasoro algebra are graded. 
In the $N=1$ case there was only one graded Ramond irreducible sector for every value of the central charge, \ie the one corresponding to the minimal lowest conformal energy $h=c/24$, and the index pairing provides no insight there if one follows the strategy that we take here for the $N=2$ super-Virasoro nets. The situation changes if one considers the different, but related, recent constructions in \cite{CHL} which also allow us to separate certain sectors of $N=1$ superconformal nets.
Hence, in this paper and in \cite{CHL}, noncommutative geometry is used for the first time to separate representations of conformal nets.
In contrast to \cite{CHL} however, the analysis given here is deeply related to and crucially relies on the rich structure of the $N=2$ superconformal context.

Our last Sect. \ref{sec:chiral-fusion} is dedicated to the study of the chiral ring (for the minimal models) from an operator algebraic point of view. 
The chiral ring associated to the $N=2$ super-Virasoro net with central charge $c$ (here we assume $c<3$) is defined and generated here by the chiral sectors, a certain subset of Neveu-Schwarz sectors, and the monoidal product by means of truncated fusion rules, hence without direct reference to pointlike localized fields. However, the algebraic structure of the chiral ring coincides with the one provided by the operator product expansion of chiral primary fields. The spectral flow (at a specific value) is known to connect the chiral sectors with the Ramond vacuum sectors, \ie those with $h=c/24$, and we illustrate and exploit this in our setting, including some comments and hints for future work. Moreover, we show that if one restricts to the family of Ramond vacuum sectors (always assuming $c<3$) then the results in Sect. \ref{sec:JLO} can be interpreted in terms of the noncommutative geometry of a finite-dimensional abelian $*$-algebra.

\section{Preliminaries on Graded-Local Conformal Nets}\label{sec:prelim}

We provide here a brief summary on graded-local conformal nets, just as much as we need to understand the general construction in the subsequent sections. The concept is a generalization of the notion of local conformal nets and has been 
explicitly introduced and worked out in \cite[Sect. 2,3,4]{CKL} under the name of Fermi conformal nets (in the case of nontrivial grading); we refer to that paper for more details, cf. also \cite{CHL}.

Let 
$\S=\{z\in\C: |z| =1\}$ be the unit circle, let $\Diff$ be the infinite-dimensional (real) Lie group of orientation-preserving smooth diffeomorphisms of $\S$ and denote by $\Diff^{(n)}$, $n\in \N \cup \{ \infty \}$, the corresponding $n$-cover. In particular $\Diff^{(\infty)}$ is the universal covering group of $\Diff$.  For $g\in \Diff$ and $z\in \S$ we will often write $gz$ instead of $g(z)$. 
By identifying the group $\PSL$ with the group of M\"{o}bius transformations on $\S$ we can consider it as a three-dimensional Lie subgroup of $\Diff$. We denote by 
$\PSL^{(n)} \subset \Diff^{(n)}$, $n\in \N \cup \{ \infty \}$, the corresponding $n$-cover so that $\PSI$ is the universal covering group of $\PSL$. We denote by  $\dot{g} \in \Diff$ the image of $g \in \Diff^{(\infty)}$ under the covering map. Since the latter restricts to the covering map of $\PSI$ onto $\PSL$ we have $\dot{g} \in \PSL$ for all $g \in \PSI$.

Let $\I$ denote the set of nonempty and non-dense open intervals of $\S$. For any $I\in \I$, $I'$ denotes 
the interior of $\S \setminus I$. Given $I \in \I$, the subgroup $\Diff_I$ of diffeomorphisms localized in $I$ is defined 
as the stabilizer of $I'$ in $\Diff$ namely the subgroup of $\Diff$ whose elements are the diffeomorphisms acting trivially 
on $I'$. Then, for any $n \in \N \cup \{ \infty \}$, $\Diff^{(n)}_I$ denotes the connected component of the identity 
of the pre-image of $\Diff_I$ in $\Diff^{(n)}$ under the covering map. We denote by $\I^{(n)}$ the set of intervals in the $n$-cover 
$\S^{(n)}$ of $\S$ which map to an element in $\I$ under the covering map. Moreover, we often identify $\R$ with $\S\setminus\{-1\}$ by means of the Cayley transform, and we write $\I_\R$ (or $\bar{\I}_\R$) for the set of bounded open intervals (or bounded open intervals and open half-lines, respectively) in $\R$.

\begin{definition}\label{def:CFT-net}
 A \emph{graded-local conformal net $\A$ on $\S$} is a map $I \mapsto \A(I)$ from the set of intervals $\I$ to the set of von Neumann algebras acting on a common infinite-dimensional separable Hilbert space $\H$ which satisfy the 
following properties:
\begin{itemize}
 \item[(A)] \emph{Isotony.} $\A(I_1)\subset \A(I_2)$ if $I_1,I_2\in\I$ and $I_1\subset I_2$.
\item[(B)] \emph{M\"{o}bius covariance.} There is a strongly continuous unitary representation $U$ of $\PSI$ such that
\[
 U(g)\A(I)U(g)^* = \A(\dot{g}I), \quad g\in \PSI , I\in\I .
\]
\item[(C)] \emph{Positive energy.} The {\it conformal Hamiltonian} $L_0$ (\ie the selfadjoint generator of the
restriction of the $U$ to the lift to $\PSI$ of the one-parameter anti-clockwise rotation subgroup of $\PSL$) is positive.
 \item[(D)] \emph{Existence and uniqueness of the vacuum.} There exists a $U$-invariant vector $\Omega\in\H$ which is unique up to a phase and cyclic for $\bigvee_{I\in\I} \A(I)$, the von Neumann algebra generated by the algebras $\A(I)$, $I\in \I$.
 \item[(E)] \emph{Graded locality.} There exists a selfadjoint unitary $\Gamma$ (the grading unitary) on $\H$ satisfying 
 $\Gamma \A(I) \Gamma = \A(I)$ for all $I\in \I$ and $\Gamma \Omega =\Omega$ and such that
 \[
 \A(I')\subset Z \A(I)'Z^*,\quad I\in\I ,
\]
where
\[
 Z:=\frac{\unit - \rmi\Gamma}{1-\rmi}.
\]
\item[(F)] \emph{Diffeomorphism covariance.} There is a strongly continuous projective unitary representation of 
$\Diff^{(\infty)}$, denoted again by $U$, extending the unitary representation of $\PSI$ and such that
\[
 U(g)\A(I)U(g)^* = \A(\dot{g}I), \quad g\in \Diff^{(\infty)} , I\in\I,
\]
and
\[
 U(g) x U(g)^* = x,\quad x\in \A(I'), g\in \Diff^{(\infty)}_I, I\in\I.
\]
\end{itemize}
A \emph{local conformal net} is a graded-local conformal net with trivial grading $\Gamma=\unit$. The \emph{even subnet}\index{even subnet} of a graded-local conformal net $\A$ is defined as the fixed point subnet $\A^\gamma$, with grading gauge automorphism $\gamma=\Ad\Gamma$. It can be shown that the projective representation $U$ of $\Diff^{\infty}$ commutes with $\Gamma$, cf. \cite[Lemma 10]{CKL}. Accordingly the restriction of  $\A^\gamma$ to the even subspace of $\H$ is a local conformal net with respect to the restriction to this subspace of the projective representation $U$ of $\Diff^{\infty}$. 
\end{definition}

Some of the consequences  \cite{CKL,FG,GL,CW05} of the preceding definition are:

\begin{itemize}
\item[$(1)$] \emph{Reeh-Schlieder Property.} $\Omega$ is cyclic and separating for every $\A(I)$, $I\in\I$.
\item[$(2)$] \emph{Bisognano-Wichmann Property.} Let $I\in \I$ and let $\Delta_I$, $J_I$ be the modular operator and the modular conjugation of $\left(\A(I), \Omega \right)$.  Then we have 
\[
 U(\delta_I(-2\pi t)) = \Delta_I^{\rmi t}, \quad t\in \R .
\]

Moreover the unitary representation $U:\PSI \mapsto B(\H)$ extends to an (anti-) unitary representation of  
$\PSL \rtimes \Z/2$  determined by 
\[
U(r_I)=Z J_I 
\]
and acting covariantly on $\A$.
Here $(\delta_I(t))_{t\in\R}$ is (the lift to $\PSI$ of) the one-parameter dilation subgroup of $\PSL$ with respect to $I$ and $r_I$ the reflection of the interval $I$ onto the complement $I'$.
\item[$(3)$] \emph{Graded Haag Duality.} $\A(I')=Z\A(I)'Z^*$, for $I\in\I$.
\item[$(4)$] \emph{Outer regularity.} 
\[
 \A(I_0) = \bigcap_{I\in\I,I\supset \bar{I_0}} \A(I),\quad I_0\in\I.
\]
\item[$(5)$] \emph{Additivity.} If $I = \bigcup_{\alpha} I_\alpha$ with $I, I_\alpha \in \I$, then 
$\A(I) = \bigvee_\alpha \A(I_\alpha)$.
 \item[$(6)$] \emph{Factoriality.} $\A(I)$ is a type $III_1$-factor, for $I\in\I$.
 \item[$(7)$] \emph{Irreducibility.} $\bigvee_{I\in\I} \A(I)= B(\H)$.
 \item[$(8)$] \emph{Vacuum Spin-Statistics theorem.} $\rme^{\rmi 2\pi L_0} =\Gamma$, in particular $\rme^{\rmi 2\pi L_0}=\unit$ for local nets, where $L_0$ is the infinitesimal generator from above corresponding to rotations. Hence the representation $U$ of $\PSI$ factors through a representation of $\PSL^{(2)}$ ($\PSL$ in the local case) and consequently its extension $\Diff^{(\infty)}$ factors through a projective representation of $\Diff^{(2)}$  ($\Diff$ in the local case). 
 \item[$(9)$] \emph{Uniqueness of Covariance.} For fixed $\Omega$, the strongly continuous projective representation $U$ of $\Diff^{(\infty)}$ making the net covariant is unique.
\end{itemize}

In the sequel, $G$ stands for either one of the two groups $\PSL$ or $\Diff$. From time to time we shall need \emph{covering nets} of a given (graded-)local conformal net. By this we mean the following:  a \emph{$G$-covariant net over $\S^{(n)}$} is a family $(\A_n(I))_{I\in\I^{(n)}}$ such that
\begin{itemize}
 \item[-]  $\A_n(I_1)\subset \A_n(I_2)$ if $I_1,I_2\in\I^{(n)}$ and $I_1\subset I_2$;
 \item[-] there is a strongly continuous unitary representation $U$ of $G^{(\infty)}$ on $\H$ such that
\[
 U(g)\A_n(I) U(g)^* = \A_n(\dot{g} I), \quad g\in\PSI, I\in\I^{(\infty)}.
\]
\end{itemize}

A representation of $\A$ is a family $\pi=(\pi_I)_{I\in\I}$ of representations
$\pi_I : \A(I) \ra B(\H_\pi)$, $I\in\I$, on a common Hilbert space $\H_\pi$ which is compatible with isotony, \ie such that $\pi_{I_2}|_{\A(I_1)} = \pi_{I_1}$ whenever $I_1\subset I_2$. $\pi$ is called \emph{locally normal} if every $\pi_I$ is normal. $\pi$ is called \emph{$G$-covariant} if there exists a projective unitary representation $U_\pi$ of $G^{\infty}$ on $\H_\pi$ satisfying
\[
U_\pi(g) \pi_I(x) U_\pi(g)^* = \pi_{\dot{g}I}(U(g)xU(g)^*),\quad g\in G^{\infty},\; x\in\A(I),\; I\in\I.
\]
$\pi$ has \emph{positive energy}  if it is $G$-covariant and the infinitesimal generator of the lift of the rotation subgroup in $U_\pi(G^{(\infty)})$ is positive. The unitary equivalence classes of irreducible locally normal representations are called the \emph{sectors} of $\A$. 
The identity representation $\pi_0$ of $\A$ on $\H$ is called the \emph{vacuum representation}, and it is automatically locally normal and
$\Diff^{(\infty)}$-covariant. 

The above notion of representation of a graded-local conformal net $\A$ agrees with the usual one for local conformal nets. 
In the graded-local case however it turns out to be very natural and useful to consider a larger class of (solitonic) representations.

\begin{definition}\label{def:CFT-covering-reps}
\begin{itemize}
\item[$(1)$] A \emph{$G$-covariant soliton} of a graded-local conformal net $\A$ on $\S$ is a family 
$ \pi = (\pi_I)_{I\in\bar{\I}_\R}$ of normal representations $\pi_I : \A(I) \ra B(\H_\pi)$, $I\in\bar{\I}_\R$, on a common Hilbert space $\H_\pi$, which is compatible with isotony in the sense that
 $\pi_{\tilde{I}}|_{\A(I)}= \pi_I$ if $I\subset \tilde{I}$, together with a projective unitary representation $U_\pi$ of $G^{(\infty)}$ on $B(\H_\pi)$ such that, for every $I\in\I_\R$:

\[
U_\pi(g) \pi_I(x) U_\pi(g)^*=\pi_{\dot{g}I}(U(g)xU(g)^*),\quad g\in \mathcal{U}_I , x\in\A(I), 
\]
where $\mathcal{U}_I$ is the connected component of the identity in $G^{(\infty)}$ of the open set \linebreak
$\{g \in G^{(\infty)}: \dot{g}I \in \I_\R \}$. If $U_\pi$ is a positive energy representation, namely the selfadjoint generator $L^\pi_0$
corresponding to the one parameter group of rotations has nonnegative spectrum, we say that $\pi$ has positive energy.

 If $\pi$ is a $G$-covariant soliton and the family $ \pi = (\pi_I)_{I\in\bar{\I}_\R}$ can be extended 
to $\I$ so that the extension is still compatible with isotony and satisfies
\[
U_\pi(g) \pi_I(x) U_\pi(g)^*=\pi_{\dot{g}I}(U(g)xU(g)^*),\quad g\in G^{\infty}, x\in\A(I), 
\]
for all $I \in \I$, then we say that $\pi$ is a \emph{DHR representation} of $\A$.

\item[$(2)$] A \emph{$G$-covariant general soliton} 
of $\A$ is a G-covariant soliton such that its restriction to the even subnet $\A^\gamma$
is a DHR representation.  In case $G=\Diff$, we shall simply say \emph{general soliton}. 

\item[$(3)$] We say that a $G$-covariant general soliton $\pi$ of $\A$ is \emph{graded} if there exists a selfadjoint unitary $\Gamma_\pi\in B(\H)$, commuting with the representation $U_\pi$ and such that 
\[
\Gamma_\pi \pi_I(x) \Gamma_\pi = \pi_I(\gamma(x)) , \quad x\in\A(I),\; I\in\I_\R.
\]

\item[$(4)$] A $G$-covariant graded general soliton $\pi$ of a superconformal net $\A$ is \emph{supersymmetric} if $L_0^\pi -\lambda \unit$ admits an odd square-root for some $\lambda \in \R$.

\end{itemize}
\end{definition}

\begin{remark}\label{rem:CFT-IR}
It can be shown (using a straightforward reasoning based on covariance relations) that a family $(\pi_I)_{I\in\I_\R}$ of normal representations of $\A$ which is covariant with respect to a given projective unitary representation of $G^{(\infty)}$ extends automatically from $\I_\R$ to $\bar{\I}_\R$, thus defines a $G$-covariant soliton. We shall make use of this (simplifying) fact when considering explicit $N=2$ super-Virasoro nets. Note also that the $G$-covariant solitons of $\A$ which are DHR representations in the sense of the above definition corresponds to the $G$-covariant locally normal representations of $\A$. In particular, the restriction of a $G$-covariant general soliton $\pi$ of $\A$ to the even subnet $\A^\gamma$ gives rise to a locally normal $G$-covariant representation of the latter net on 
$\H_\pi$.
\end{remark}

In various cases, as a consequence of the results in \cite{Wei06}, the positive energy condition is automatic for $G$-covariant general solitons, see \cite[Prop. 12 \& Prop. 21]{CKL}. In particular an irreducible $G$-covariant general soliton is always of positive energy.

For the more common case of a \emph{local} net $\B$ over $\S$ (like the even subnet $\B=\A^\gamma$) we recall the following associated ``global algebras'' \cite{Fred,FRS2,GL1}, see also \cite{CCHW,CHL}:

\begin{definition}\label{def:CFT-univC}
The \emph{universal C*-algebra} $C^*(\B)$ of $\B$ is determined by the following properties:
\begin{itemize}
\item[-] for every $I\in\I$, there are unital embeddings $\iota_I:\B(I)\ra C^*(\B)$, such that $\iota_{I_1|\B(I_2)} = \iota_{I_2}$ whenever $I_1\subset I_2$, and all $\iota_I(\B(I))$ together generate $C^*(\B)$ as a C*-algebra;
\item[-] for every representation $\pi$ of $\B$ on some Hilbert space $\H_\pi$, there is a unique $*$-representation $\tilde{\pi}:C^*(\B)\ra B(\H_\pi)$ such that
\[
\pi_I = \tilde{\pi}\circ \iota_I,\quad I\in \I.
\]
\end{itemize}
It can be shown to be unique up to isomorphism. Let $(\tilde{\pi}_u,\H_u)$ be the \emph{universal representation} of $C^*(\B)$: the direct sum of all GNS representations $\tilde{\pi}$ of $C^*(\B)$. Since it is faithful, $C^*(\B)$  can be identified with $\tilde{\pi}_u(C^*(\B))$. We call the weak closure $W^*(\B) = \tilde{\pi}_u(C^*(\B))''$ the \emph{universal von Neumann algebra} of $\B$. We shall drop the $\tilde{\cdot}$ sign henceforth.
\end{definition}

When no confusion can arise we will identify $C^*(\B)$ with $\tilde{\pi}_u(C^*(\B)) \subset W^*(\B)$. Moreover, we will 
write $\B(I)$ instead of $\iota_I(\B(I))$. With these conventions, for every $I \in \I$ we have the inclusions 
$\B(I) \subset C^*(\B) \subset W^*(\B)$.

Coming back to the general case of graded-local nets, a fundamental first consequence of Definition \ref{def:CFT-covering-reps} is

\begin{proposition}[{\cite[Sect. 4.3]{CKL}}]\label{prop:CKL22}
Let $\A$ be  a graded-local conformal net and $\pi$ an irreducible $G$-covariant general soliton of $\A$. Then the following three conditions are equivalent:
\begin{itemize}
 \item[-] $\pi$ is graded,
 \item[-] $\pi_{|\A ^\gamma}$ is reducible,
 \item[-] $\pi_{|\A ^\gamma} =: \pi_+\oplus \pi_- \simeq  \pi_+ \oplus \pi_+ \circ  \hat{\gamma}$, 
\end{itemize}
with $\pi_+$ and $\pi_-$ inequivalent irreducible  DHR representations of $\A^\gamma$ on the eigenspaces $\H_{\pi, \pm}$ of 
$\Gamma_\pi$ corresponding to the eigenvalues $\pm 1$ and $\hat{\gamma}$ a localized DHR automorphism of $\A^\gamma$ dual to the grading. Moreover, if $\tilde{\pi}$ is another irreducible $G$-covariant general soliton of $\A$ then $\tilde{\pi}$ is unitarily equivalent to $\pi$
if and only if $\tilde{\pi}_+$ is unitarily equivalent to $\pi_+$ or to $\pi_-$. 

\end{proposition}

\begin{proof} We prove only the last statement. For the proof of the remaining statements we refer the reader to \cite[Prop. 22]{CKL}. If 
 $\tilde{\pi}$ is unitarily equivalent to $\pi$ and $u: \H_{\pi} \to \H_{\tilde{\pi}}$ is a corresponding unitary intertwiner then 
 $u \Gamma_\pi u^* = \pm \Gamma_{\tilde{\pi}}$ so that $\tilde{\pi}_+$ is unitarily equivalent to $\pi_\pm$. Conversely, assume that 
 $\tilde{\pi}_+$ is unitarily equivalent to $\pi_\pm$. Then  $\tilde{\pi}_- \simeq \pi_+\circ\hat{\gamma}$ is unitarily equivalent to 
 $\pi_\mp \simeq \pi_\pm \circ \hat{\gamma}$ and hence $\pi_+ \oplus \pi_- \simeq \tilde{\pi}_+ \oplus \tilde{\pi}_-$. It follows that the commutant of $\bigcup_{I \in \I_\R} \left( (\pi_I \oplus \tilde{\pi}_I)(\A^\gamma(I) )    \right)$ is an eight-dimensional algebra and hence 
 $\pi$ cannot be inequivalent to $\tilde{\pi}$, otherwise this commutant would be a four-dimensional algebra.
 \end{proof}

It has been shown in \cite[Prop. 22]{CKL} that for irreducible graded $\pi$ and under the assumption of finite statistical dimension on 
$\pi_+$, we have, up to unitary equivalence, the two possibilities
\[
 \rme^{\rmi 2\pi L_0^\pi} =  \rme^{\rmi 2\pi L_0^{\pi_+}}\oplus \pm \rme^{\rmi 2\pi L_0^{\pi_+}},
\]
so, $\rme^{\rmi 4\pi L_0^\pi}$ is a scalar, because, $\pi_+$ being irreducible, $\rme^{\rmi 2\pi L_0^{\pi_+}}$ is a scalar. 

Here `` $+$ " will correspond to $(R)$ in the following theorem, `` $-$ " to $(NS)$. Thus every irreducible general soliton of finite statistical dimension factorizes through a representation of a net over $\S^{(2)}$.  
Since the statistical dimension in general may be infinite we shall assume this factorization property from the beginning in the definition of Neveu-Schwarz and Ramond representations of a graded-local conformal net, \cf Theorem \ref{th:CFT-R-NS} here below.

\begin{lemma}\label{lem:CFT-DHR}
If a $G$-covariant soliton $\pi$ on a graded-local conformal net $\A$ is such that $\rme^{2\pi\rmi L_0^\pi}$ is either a scalar or implements the grading, then $\pi_{|\A^\gamma}$ is a DHR representation, \ie $\pi$ is a $G$-covariant \emph{general} soliton of $\A$.
\end{lemma}

\begin{proof}
Let $U_\pi$ be the covariance unitary representation of $\pi$. We can extend $\pi$ to a representation of the promotion $\A^{(\infty)}$ to the universal cover $\S^{(\infty)}$ by setting $\pi_{gI}:= {\rm Ad}U_\pi(g)\cdot\pi_I$ for every $I\in\I^{(\infty)}$.
As $U_\pi(4\pi)=\rme^{4\pi\rmi L_0^\pi}$ commutes with the image of $\pi$,
$\pi$ defines actually a $G$-covariant representation of the double cover net $\A^{(2)}$ over $\S^{(2)}$.
By assumption,  $U_\pi(2\pi)$ commutes with the image of the restriction of $\pi$ to the even subnet $\A^\gamma$ of $\A$, so $\pi$ is a DHR representation of $\A^\gamma$, \cf
 \cite[Prop. 19]{CKL}.
\end{proof}

\begin{theorem}[{\cf \cite[Sect. 4.3]{CKL}}]\label{th:CFT-R-NS}
Let $\A$ be a graded-local conformal net over $\S$
and let $\pi$ be an irreducible $G$-covariant general soliton of $\A$  such that $\rme^{\rmi 4\pi L_0^\pi}$ is a scalar, and denote $\pi|_{\A^\gamma}=:\pi_+ \oplus \pi_+\hat{\gamma}$ or $\pi|_{\A^\gamma}=:\pi_+$ with an irreducible representation $\pi_+$ of $\A^\gamma$ (for graded or ungraded $\pi$, respectively). Then $\pi$ is of either of the subsequent two types:
\begin{itemize}
\item[$(NS)$] $\pi$ is actually a DHR representation of $\A$; equivalently,\\
$\rme^{\rmi 2\pi L_0^\pi}$ implements the grading.
\item[$(R)$]   $\pi$ is not a DHR representation but only a general soliton of $\A$; equivalently,\\
$\rme^{\rmi 2\pi L_0^\pi}$ is a scalar, hence does not implement the grading.
\end{itemize}
In case $(NS)$, $\pi$ is called a \emph{Neveu-Schwarz representation} of $\A$, and in case $(R)$, a \emph{Ramond representation} of 
$\A$, the latter being however actually only a general soliton, \ie a representation of $\A^{(2)}$ over $\S^{(2)}$, and not a proper representation of $\A$. A direct sum of irreducible Neveu-Schwarz (Ramond) representations is again called a Neveu-Schwarz (Ramond) representation. 
\end{theorem}

\section{$N=2$ Super-Virasoro Nets and Their Representations}\label{sec:SVir-net}

\begin{definition}\label{def:SVir2-alg}\index{super-Virasoro algebra $\SVir^N$}
For any $t\in\R$, the {\em $N=2$ super-Virasoro algebra} $\SVir^{N=2,t}$ is the infinite-dimensional Lie superalgebra generated by linearly independent even elements $L_n,J_n$ and odd elements $G_r^\pm$, where $n\in\Z$, $r\in\frac12 \mp t + \Z$, together with an even central element $\hat{c}$ and with (anti-) commutation relations
\begin{align*}
[L_m , L_n] &= (m-n)L_{m+n} + \frac{\hat{c}}{12}(m^3 - m)\de_{m+n, 0},\\
[L_m, G_r^\pm] &= \Big(\frac{m}{2} - r \Big)G_{m+r}^\pm,\\
[G_r^+, G_s^-] &= 2L_{r+s} + (r-s) J_{r+s} + \frac{\hat{c}}{3} \Big(r^2 - \frac14 \Big)\de_{r+s,0},\\
[G_r^+, G_s^+] &= [G_r^-, G_s^-] = 0,\\
[L_m , J_n] &= -n J_{m+n},\\
[G_r^\pm, J_n] &= \mp G_{r+n}^\pm,\\
[J_m,J_n] &= \frac{\hat{c}}{3} m \delta_{m+n,0}.
\end{align*}
The \emph{Neveu-Schwarz (NS) $N=2$ super-Virasoro algebra} is the super-Virasoro algebra with $t=0$, while the \emph{Ramond (R) $N=2$ super-Virasoro algebra} is the one with $t=1/2$. Sometimes we shall write simply $\SVir^{N=2}$ for the Neveu-Schwarz $N=2$ super-Virasoro algebra $\SVir^{N=2,0}$.
\end{definition}

For $t\in \frac12 \Z$, we have  $\frac12 + t + \Z = \frac12 - t + \Z$, hence,  instead of $G^\pm_r$, one can consider the modes
\[
G^1_r:=\frac{G^+_r + G^-_{r}}{\sqrt{2}},\quad G^2_{r}:=-\rmi \frac{(G^+_{r} - G^-_{r})}{\sqrt{2}},
\]
and we shall use  them frequently in the following. 

Moreover for all $t\in \R$, the Lie superalgebra $\SVir^{N=2,t}$ is equipped with a natural anti-linear involution ($*$-structure), such that the adjoints of $L_n$, $J_n$, $G^{\pm}_r$ are respectively $L_{-n}$, $J_{-n}$ $G^{\mp}_{-r}$, and $\hat c$ is selfadjoint.

If $t-s \in \Z$ we have  $\frac12 \mp t + \Z = \frac12 \mp s + \Z$ and the Lie superalgebra $\SVir^{N=2,t}$ is trivially isomorphic (i.e. equal) to $\SVir^{N=2,s}$ through the linear map defined by $L_n \mapsto L_n$, $J_n \mapsto J_n$, 
$G^{\pm}_r \mapsto G^{\pm}_r$, $n \in \Z$, $r \in \frac12 \mp t + \Z$. As we will see in Sect. \ref{sec:spfl}, cf. Proposition  
\ref{prop:spfl-alg-iso}, 
the algebras  $\SVir^{N=2,t}$, $t\in \R$, are all ismorphic trough the spectral flow although the isomorphisms considered there $\eta'_s \circ \eta_t : \SVir^{N=2,t} \to \SVir^{N=2,s}$, $t,s \in \R$, do not preserve the generators unless $t=s$ in contrast to the trivial isomorphism considered above in the case $t-s \in \Z$. It should be clear from the above discussion that the reason to consider all the isomorphic Lie superalgebras $\SVir^{N=2,t}$  is not motivated by the corresponding Lie superalgebra structure but by the special choice of the generators.

We are interested in linear vector space representations. We restrict here to the case $t\in \frac12 \Z$, i.e. to the NS and R case. These representations should satisfy the usual conditions explained in \cite[Sect. 4]{CHKL}, in short, they should be unitary (i.e. $*$-preserving) with respect to a suitable scalar product turning the vector space into a pre-Hilbert space, $\hat{c}$ should be represented by a real scalar $c$ (the {\it central charge}), and $L_0$ should be diagonalizable with every eigenspace finite-dimensional and only nonnegative eigenvalues. 

Note that in the NS case the positivity of $L_0$ follows automatically from the commutation relations $2L_0  = [G^i_{1/2}, G^i_{-1/2}]$, $i=1,2$. In the R case we have $2L_0 - \frac{c}{12}\unit = [G^+_0, G^-_0] \geq 0$ and hence $L_0$ is bounded from below. It then follows by unitarity that $c\geq 0$ \footnote{If  $\psi$ is an eigenvector of $L_0$ then
$0 \leq (L_{-n}\psi,L_{-n}\psi)= (\psi,[L_n,L_{-n}]\psi)=2n(\psi,L_0\psi) + \frac{c}{12}(n^3-n)$ for all sufficiently large positive integers  $n$. Hence $c$ must be a non-negative real number. }
and $L_0\geq \frac{c}{24}\unit \geq 0$.
Accordingly an irreducible unitary representation is completely determined by the corresponding irreducible
unitary representation of the zero modes on the lowest energy subspace (the subspace of highest weight vectors).
In the NS case the algebra of zero modes is abelian and irreducibility implies that the lowest energy subspace is
one-dimensional and spanned by a single vector $\Omega_{c,h,q}$ of norm one such that
$L_0  \Omega_{c,h,q}  = h \Omega_{c,h,q}$ and $J_0  \Omega_{c,h,q}  = q\Omega_{c,h,q}$. The real numbers
$c,h,q$ completely determine the representation (up to unitary equivalence). In the R case the algebra of zero modes is non-abelian and there are two possibilities. If the lowest energy $h$ is equal to $c/24$, the lowest energy subspace must be one dimensional again, spanned by a normalized common eigenvector  $\Omega_{c,h,q}$ of $L_0$ and $J_0$ with eigenvalues $h$ and $q$, respectively, and satisfying $G_0^+ \Omega_{c,h,q} = G_0^- \Omega_{c,h,q} =0$. In contrast if $h > c/24$, the lowest energy subspace must be two-dimensional. Then one can choose a common normalized eigenvector $\Omega^-_{c,h,q}$ for $L_0$ and $J_0$, with eigenvalues $h$ and $q$ respectively by imposing the supplementary condition $G_0^+  \Omega^-_{c,h,q}=0$. The lowest energy subspace is spanned by $\Omega^-_{c,h,q}$ and $\Omega^+_{c,h,q-1}$ where $\Omega^+_{c,h,q-1}= (2h-c/12)^{-\frac12}G_0^-\Omega^-_{c,h,q}$ is normalized and satisfies $L_0\Omega^+_{c,h,q-1}=h\Omega^+_{c,h,q-1}$ and $J_0\Omega^+_{c,h,q-1}=(q-1)\Omega^+_{c,h,q-1}$. With the above convention the numbers $c,h,q$ completely determine the representation up to unitary equivalence also in the R case. 
Both in the NS and the R case we shall sometime use the more explicit notation $c_\pi,h_\pi,q_\pi$ instead of $c, h, q$ for the numbers characterizing the unitary representation $\pi$.  
As in the cases $N=0, 1$, unitarity gives restrictions on the possible values of $c,h,q$.
The situation is described in

\begin{theorem}[\cite{BFK}, \cite{DVYZ}, \cite{Iohara}]\label{th:SVir2-reps}
For any irreducible unitary representation of the Neveu-Schwarz $N=2$ super-Virasoro algebra $\SVir^{N=2,0}$
the corresponding values of $c,h,q$ satisfy one of the following conditions:
\begin{itemize}
\item[NS1] $c\geq 3$ and $2h-2nq+(\frac{c}{3}-1)(n^2- \frac14) \geq 0$ for all $n\in \frac12 +\Z$.
\item[NS2] $c\geq 3$ and $2h-2nq+(\frac{c}{3}-1)(n^2- \frac14) = 0$, \\
$2h-2(n +\sgn(n))q + (\frac{c}{3}-1)\left[(n+\sgn(n))^2- \frac14\right] < 0$
for some $n\in \frac12 +\Z$ and 
$2(\frac{c}{3}-1)h - q^2 +\frac{c}{3} \geq 0$.
\item[NS3]  $c= \frac{3n}{n+2},\quad h=\frac{l(l+2)-m^2}{4(n+2)},\quad q=-\frac{m}{n+2}$,
where $n, l, m \in\Z$ satisfy $n\geq 0$, $0\le l\le n$, $l+m\in 2\Z$ and $|m| \leq l$.
\end{itemize}
For any irreducible unitary representation of the Ramond $N=2$ super-Virasoro algebra $\SVir^{N=2,\frac12}$
the corresponding values $c,h,q$ satisfy one of the following conditions:
\begin{itemize}
\item[R1] $c\geq 3$ and $2h-2n(q-\frac12)+(\frac{c}{3}-1)(n^2- \frac14) - \frac14 \geq 0$ for all $n\in \Z$
\item[R2] $c\geq 3$ and $2h-2n(q-\frac12)+(\frac{c}{3}-1)(n^2- \frac14) -\frac14 = 0$, \\
$2h-2(n +\sgn(n-\frac12))(q-\frac12) + (\frac{c}{3}-1)\left[(n+\sgn(n-\frac12))^2- \frac14\right] - \frac14 < 0$
for some $n\in \Z$ and 
$2(\frac{c}{3}-1)(h-\frac{c}{24}) - (q-\frac12)^2 + \frac{c}{3}\geq 0$ .
\item[R3] $c= \frac{3n}{n+2},\quad h=\frac{l(l+2)-m^2}{4(n+2)} +\frac{1}{8},\quad q=-\frac{m}{n+2} +\frac{1}{2}$,
where $n, l, m \in\Z$ satisfy $n\geq 0$, $0\le l\le n$, $l+m+1\in 2\Z$ and $|m-1| \leq l$.
\end{itemize}
Conditions NS1, NS3, R1, R3, are also sufficient, namely if values $c,h,q$ satisfy one of them then there exists a corresponding irreducible unitary representation. In particular all the values in the discrete series of representations
(conditions NS3 and R3) with $c=3n/(n+2)$ are realized by the coset construction for the inclusion
$\Uone_{2n+4}\subset \SUtwo_n \otimes \CAR^{\otimes 2}$ for every nonnegative integer $n$. Here $\CAR^{\otimes 2}$ denotes the theory generated by two real chiral free Fermi fields.
\end{theorem}

For every allowed value of $c$, there is a corresponding unique representation of the  Neveu-Schwarz  $N=2$ super-Virasoro algebra 
$\SVir^{N=2,0}$ with $h=q=0$, the \textbf{vacuum representation} with central charge $c$. Actually it can be shown, \eg using 
Theorem \ref{th:SVir2-reps}, that, for an irreducible unitary representation of the Neveu-Schwarz  $N=2$ super-Virasoro algebra, 
$h=0$ implies $q=0$ so that the vacuum representation with central charge $c$ is the unique irreducible unitary representation with lowest energy $h=0$.

Moreover, for every allowed value of $c$ there are irreducible representations of the Ramond $N=2$ super-Virasoro algebra with $h=c/24$ (and the corresponding allowed values for $q$) and we call them \textbf{Ramond vacuum representations} with central charge $c$.
Notice from Definition \ref{def:SVir2-alg} that in contrast to the case $N=1$, every irreducible representation $\pi$ is automatically graded by $\Gamma_\pi: = \rme^{-\rmi \pi q}\rme^{\rmi \pi J_0^\pi}$ as follows easily from the $N=2$ super-Virasoro algebra commutation relations. Moreover, as for $N=1$ the irreducible representations of the Neveu-Schwarz  $N=2$ super-Virasoro algebra are graded by 
$\Gamma_\pi: = \rme^{-\rmi 2\pi h}\rme^{\rmi 2\pi L_0^\pi}$. In the following we shall use this standard choice so that there is always an even 
lowest energy vector. Sometimes, however, we shall say that $\rme^{\rmi \pi J_0^\pi}$, or  $\rme^{\rmi 2\pi L_0^\pi}$, implements the grading. Note that the latter operators are not selfadjoint and hence do not define grading operators in the precise sense of 
Definition \ref{def:CFT-covering-reps}.

Our next goal in this section is to define a net associated to the vacuum representation with central charge $c$ of the Neveu-Schwarz $N=2$ super-Virasoro algebra for every allowed value of $c$. This will be done in the standard way by using certain unbounded field operators (in the vacuum representation). Moreover, we will give a description, based on the results in  \cite{CW}, of the representation theory of these nets in terms of the unitary representations with central charge $c$ of the Neveu-Schwarz and Ramond Virasoro algebras super-Virasoro algebra. For these reasons we will need to consider analogous field operators in the latter representations.
 
Let $\pi$ be an irreducible unitary (and hence positive energy) representation of the Neveu-Schwarz or Ramond $N=2$ super-Virasoro algebra as above with central charge $c_\pi=c$. 
If $\H_\pi$ denotes the Hilbert space completion of the corresponding representation space, we shall freely say, with some abuse of language, that $\pi$ is a representation on $\H_\pi$. Then we denote the generators in that representation by $L_n^\pi, G_r^{i,\pi},J_n^\pi$, $n\in \Z$, $r\in \frac12 + \Z$ in the NS case or $r\in \Z$ in the R case. These generators can be considered as densely defined operators on the separable Hilbert space $\H_\pi$ which are closable as a consequence of unitarity. We shall  often denote their closure by the same symbols. In particular we consider $L_0^\pi$ as a selfadjoint operator on $\H_\pi$. It has pure point spectrum and in fact the dense subspace $\H_\pi^f$ spanned by its eigenvectors coincides with the original representation space. Moreover, it follows from \cite{BFK} that $\rme^{-\beta L_0}$ is a trace class operator for all $\beta >0$. 

 If $\pi=\pi_0$ is the (NS) vacuum representation with central charge $c$, we shall drop the superscript $\pi$.

We shall need the following estimates ({\it energy bounds}), cf. \cite{BSM,CKL,CHKL}. Let $c$ be any of the allowed values of the central charge for the irreducible unitary representations of $\SVir^{N=2,0}$ (or equivalently of $\SVir^{N=2,\frac12}$).
Then, there is a constant $M_c>0$ such that, for any irreducible unitary representation $\pi$ with central charge $c$ of the Neveu-Schwarz (resp. Ramond) $N=2$ super-Virasoro algebra we have

 \begin{align}\label{eq:LEB1}
 \|L_m^\pi \psi\| \leq& M_c (1+|m|^{\frac{3}{2}})\|(\unit+L_0^\pi)\psi\|,\nonumber\\
 \|G_r^{i,\pi} \psi\|\leq& (2+ \frac{c}{3}r^2)^{\frac12}\|(\unit+L_0^\pi)^{\frac12}\psi\|,\\
 \|J_m^\pi \psi\|\leq& (1+ c|m|)^{1/2}\|(\unit+L_0^\pi)^{\frac12}\psi\|, \nonumber
 \end{align}
 for all $\psi\in \H_\pi^f$, $m\in\Z$, $r\in \frac12 + \Z$ (resp. $r\in \Z$). 

As a consequence of the above estimates, together with the commutation relations in Definition \ref{sec:SVir-net}, the subspace $\Cci(L_0)$ of $\Cci$ vectors of $L_0$ is a common invariant core for the closed operators $L_n^\pi, G_r^{i,\pi},J_n^\pi$, $i=1,2$, $n\in \Z$, $r\in \frac12 + \Z$ (resp. $r\in \Z$).

We denote by $\Cci(\S):=\Cci(\S,\C)$ the space of complex-valued smooth functions on $\S$ and, for any open interval $I\subset \S$, 
we denote by $\Cci_c(I) \subset \Cci(\S)$ the subspace of complex-valued functions with support contained in $I$. We will also denote
by $\Cci(\S,\R)$ and by $\Cci_c(I,\R)$ the corresponding real subspaces of real-valued functions. Moreover, for $f \in \Cci(\S)$ and $s\in \R$ we write $f_s:= \frac{1}{2\pi}\int_{-\pi}^{\pi} \rme^{-\rmi s\theta} f(\rme^{\rmi \theta})\rmd \theta$.

We first assume that $\pi$ is an irreducible unitary representation with central charge $c$ of the Ramond $N=2$ super-Virasoro algebra. 

For any function $f\in\Cci(\S)$ the Fourier coefficients $f_n$, $n\in \Z$, are rapidly decreasing, and by standard arguments based on the above energy bounds , the series
\[
 \sum_{n\in\Z} f_n L_n^\pi,\quad \sum_{r\in \Z} f_r G^{i,\pi}_r, \quad \sum_{r\in \Z} f_r G^{\pm,\pi}_r, 
 \quad \sum_{n\in\Z} f_n J_n^\pi, \quad i=1,2,
\]
give rise to densely defined closable operators with common invariant core $\Cci(L^\pi_0)$.  We denote their closures, the so-called smeared fields, by $L^\pi(f)$, $G^{i,\pi}(f)$, $i=1,2$,  $G^{\pm, \pi}(f)$ and $J^\pi(f)$ respectively. Moreover, as a further consequence of the energy bounds it can be shown that 

\[
L^\pi(f)^*= L^\pi(\overline{f}) ,\quad G^{i,\pi}(f)^* = G^{i,\pi}(\overline{f}),\; i=1,2,\quad  G^{+, \pi}(f)^* = G^{-, \pi}(\overline{f}),\quad 
J^\pi(f)^*= J^\pi(\overline{f}),
\]
so that, if $f \in \Cci(\S,\R)$ is real, then $L^\pi(f)$, $G^{i,\pi}(f)$, $i=1,2$ and  $J^\pi(f)$ are essentially selfadjoint on $\Cci(L^\pi_0)$ (in fact on any core for $L^\pi_0$), cf. \cite[Sect. 2]{BSM}, \cite[Sect. 2 ]{Carpi}, \cite[Sect. 4]{CW05} and \cite[Sect. 4]{CHKL}.

Now, if $\pi$ is an irreducible unitary representation of the Neveu-Schwarz $N=2$ super-Virasoro algebra, the smeared fields 
$L^\pi(f)$, $J^\pi(f)$, $f\in \Cci(\S)$ can be defined in the same way and have the properties discussed in the Ramond case. 
In particular, for $f\in \Cci(\S,\R)$ they are essentially selfadjoint on any core for $L^\pi_0$. We would also like to define 
$G^{i,\pi}(f)$, $i=1,2$ and $G^{\pm,\pi}(f)$ through the series
\[
\sum_{r\in \frac12+\Z} f_r G^{i,\pi}_r, \quad \sum_{r\in \frac12+\Z} f_r G^{\pm,\pi}_r,  \quad i=1,2,
\]
but, for $f\in \Cci(\S)$, the Fourier coefficients $f_r$, $r\in \frac12+\Z$ are, in general, not rapidly decreasing. This problem, however, can be overcome by restricting to functions $f\in \Cci_c(\S\setminus \{-1\})$. Accordingly, for the latter functions, closed operators $G^{i,\pi}(f)$, $i=1,2$ and $G^{\pm,\pi}(f)$ with the properties discussed in the Ramond case can also be defined in the Neveu-Schwarz case. In particular, for any real function $f\in \Cci_c(I,\R)$, the closed operators $G^{i,\pi}(f)$, $i=1,2$ are essentially selfadjoint on 
$\Cci(L^\pi_0)$.

In order to define the $N=2$ super-Virasoro nets we consider the special case when $\pi=\pi_0$ is the (NS) vacuum representation with central charge $c$. As said before, in this case we shall drop the superscript $\pi$.   We can define, similarly to 
\cite[Sect. 6.3]{CKL},  an isotonous net of von Neumann algebras over $\R$ on the Hilbert space $\H = \H_{\pi_0}$ by

\begin{equation}\label{def:SVir2-net}
\A_c(I) = \{ \rme ^{\rmi J(f)},\rme ^{\rmi L(f)},\rme ^{\rmi G^{i}(f)}:\; f\in \Cci_c(I,\R),\; i=1,2 \}'',\; I \in \I_\R.
\end{equation}

The extension of the net $\A_c$ to $S^1$, whose existence is guaranteed by the following theorem, will be denoted again by $\A_c$ and called the {\it $N=2$ super-Virasoro net} with central charge $c$.

\begin{theorem}\label{th:SVir2-diff-cov}
The family $(\A_c(I))_{I\in\I_\R}$ extends to a graded-local conformal net $\A_c=(\A_c(I))_{I\in\I}$ over $\S$.
 
\end{theorem}

\begin{proof}
The proof goes in complete analogy to the one of \cite[Thm. 33]{CKL} for the $N=1$ super-Virasoro nets. Here, for the reader's convenience, we only recall the main points and discuss some of the adaptations which are needed in the $N=2$ case.
First of all, graded locality is consequence of the fact that
\[
[X(f),Y(g)]= 0, \quad \supp(f) \cap \sup (g) =\emptyset,
\]
for $X$ and $Y$ any of the fields $L,G^i,J$ and of the energy bounds \eqref{eq:LEB1} together with the adaptation of the arguments in 
\cite[Sect. 2]{BSM} based on \cite[Thm. 3.2]{Driessler}. Concerning covariance, we first define a projective unitary representation of 
$\Diff^{(\infty)}$ (in fact of $\Diff^{(2)}$) by integrating the projective representation of the corresponding Lie algebra of smooth vector fields on $\S$ associated to the representation of the Virasoro algebra on $\H_\pi$ defined by the operators $L_n$, $n\in \Z$, see \cite{Tol} and \cite{GW2}. Then we have to show that all $J,G^i,L$ transform covariantly with respect to the restriction of that representation to $\PSI$.  In fact, going through all the steps in the proof of \cite[Thm. 33]{CKL}, we just have to notice that (in the notation used there)
\[
- {\rmi} \frac{\rmd}{\rmd t}(J(\beta_{1}(\exp{(tf_1)})f_2)\psi_0)|_{t=0} = [L(f_1),J(f_2)]\psi_0,
\]
where $\psi_0\in\Cci(L_0)$, $f_1,f_2\in\Cci(\S,\R)$,

\[
\left(\beta_{ 1}(g)f_i \right)(z) := f_i(g^{-1}(z)), \quad g\in \PSL, z\in \S, i=1,2
\]
and $\exp{(tf_1)}$ is the one-parameter group of diffeomorphisms generated by the vector field $f_1 \frac{\rmd}{\rmd \theta}$. 
Once this is done, we show cyclicity of the vacuum vector $\Omega$ for the net, and finally extension of $\PSL$ to diffeomorphism covariance, literally as in \cite[Thm.33]{CKL}.
\end{proof}

Now that the net $\A_c$ is defined, we would like to study its representations. We shall use the following theorem from \cite{CW}
(cf. also \cite[Prop. 2.14]{CHL}).

\begin{theorem}[\cite{CW}] \label{th:SVir-pi-exp-pi}
Let $\pi=(\pi_I)_{I\in \bar{\I}_\R}$ be an irreducible Neveu-Schwarz (resp. Ramond) representation of $\A_c$. 
Then, there is a unitary irreducible representation, also denoted $\pi$, of the Neveu-Schwarz (resp. Ramond) $N=2$ super-Virasoro algebra with central charge $c$ such that
\[
 \pi_I(\rme^{\rmi L(f)}) = \rme^{\rmi L^\pi(f)}, \quad
\pi_I(\rme^{\rmi G^i(f)}) = \rme^{\rmi G^{i,\pi}(f)}, \quad
\pi_I(\rme^{\rmi J(f)}) = \rme^{\rmi J^\pi(f)}, \quad
\]
for $f\in\Cci_c(I, \R)$, $i=1,2$, $I\in\I_\R$.  In particular every irreducible Ramond representation $\pi$ of $\A_c$ is graded by 
$\Gamma_\pi :=\rme^{-i \pi q_\pi}  \rme^{i\pi J_0} $.

\end{theorem}

 From the above theorem we see that to every irreducible Neveu-Schwarz (resp. Ramond) representation of $\A_c$ there corresponds an irreducible representation of the Neveu-Schwarz (resp. Ramond) $N=2$ super-Virasoro algebra. The converse is not known in general but it can be shown to hold e.g.  when $c<3$ as a consequence of the coset construction, cf. Theorem \ref{th:SVir2-reps} and Sect. \ref{sec:coset}.  Cf. also \cite[Sect. 6.4.]{KL1} and \cite[Sect. 3]{CKL} for related results in the $N=0,1$ cases. Actually, the analogue of this problem appears to be still open for certain values of the central charge $c$ and of the lowest conformal energy $h$ also in the case of Virasoro nets ($N=0$ case), see \cite[page 268]{Carpi} and similarly for the $N=1$ super-Virasoro nets. 
 
We end this section with the definition of $N=2$ superconformal net. This is the analogue of the definition of ($N=1$) superconformal net
in \cite[Sect. 7]{CKL} and \cite[Sect. 2]{CHL}. 

\begin{definition} A {\it $N=2$ superconformal net} is a graded-local conformal net $\mathcal{A}$ with central charge $c$ containing 
the $N=2$ super-Virasoro net $\A_c$ as a covariant subnet and such that the corresponding representations of  $\Diff^{(\infty)}$ agree.
\end{definition}

Note that, since the NS $N=2$ super-Virasoro algebra contains copies of the NS ($N=1$) super-Virasoro algebra every $N=2$ 
superconformal net $\A$ is also an ($N=1$) superconformal net in the sense of \cite[Def. 2.11]{CHL}.

\section{Spectral Flow for the $N=2$ Super-Virasoro Nets}\label{sec:spfl}

A remarkable property of the $N=2$ super-Virasoro algebra is the ``homotopic'' equivalence of the Neveu-Schwarz and the Ramond algebra in the sense that there exists a deformation of one into the other, first discussed in \cite{SS}:

\begin{definition}\label{def:spfl-alg}
The \emph{spectral flow} of the Lie algebra $\SVir^{N=2}$ is the family of linear maps
$\eta_t: \SVir^{N=2, t} \ra \SVir^{N=2,0}$, with $t\in\R$, defined on the generators by
\begin{align*}
\eta_{t}(L_n) & := L_n + t J_n + \frac{{\hat c}}{6} t^2 \delta_{n,0},\quad n\in\Z \\
\eta_{t}(J_n) & := J_n + \frac{{\hat c}}{3} t \delta_{n,0}, \quad n\in\Z\\
\eta_{t}(G^\pm_r) & := G_{r\pm t}^{\pm}, \quad  r\in \mp t + \frac12 +\Z \\
\eta_{t}(\hat c) & := \hat c .
\end{align*}
In other words, the map $\eta_t$ embeds $\SVir^{N=2,t}$ into $\SVir^{N=2}$.
\end{definition}

\begin{proposition}\label{prop:spfl-alg-iso}
The linear maps $\eta_t$ are Lie superalgebra isomorphisms, so the Lie superalgebras $\SVir^{N=2,t},\ t\in \R$, are all isomorphic. In particular, the Neveu-Schwarz $N=2$ super-Virasoro algebra and the Ramond $N=2$ super-Virasoro algebra are isomorphic.
\end{proposition}

\begin{proof}
Define a linear map $\eta_t':\SVir^{N=2,0}\ra \SVir^{N=2,t}$, given on the generators by
\begin{align*}
\eta_{t}'(L_n) &:= L_n - t J_n + \frac{\hat{c}}{6} t^2 \delta_{n,0},\quad n\in\Z \\
\eta_{t}'(J_n) &:= J_n - \frac{\hat{c}}{3} t \delta_{n,0}, \quad n\in\Z\\
\eta_{t}'(G^\pm_r) &:= G_{r\mp t}^{\pm}, \quad  r\in \frac12 +\Z \\
\eta_{t}'(\hat c) & := \hat c .
\end{align*}
It is straightforward to check that it is an inverse for $\eta_t$, so we have bijectivity. Moreover, using Definition \ref{def:SVir2-alg}, one finds that $\eta_t$ is a Lie algebra homomorphism, so we are done.
\end{proof}

The main purpose of this section is to set up an operator algebraic version of the $N=2$ spectral flow. 
In order to treat the local algebras, we need to study the action of the spectral flow on smeared fields. 

Consider now $\SVir^{N=2,t}$ and denote by $\pi_0$ the vacuum representation with central charge $c$ of the Neveu-Schwarz 
$N=2$ super-Virasoro algebra $\SVir^{N=2}$, i.e. the unique irreducible unitary representation with central charge $c$ and lowest energy $h=0$ (and consequently $q=0$). We denote by  
$L_n^t := \pi_0(\eta_t(L_n))$, $J_n^t := \pi_0(\eta_t(J_n))$, $G_r^{\pm,t} := \pi_0(\eta_t(G_r^\pm))$,  the generators in the unitary 
representation $\pi_0\circ \eta_t$ of $\SVir^{N=2,t}$ (we suppress the superscript $\cdot^t$ when $t=0$). 
Similarly to the case $t=0$ considered in Sect. \ref{sec:SVir-net}, for any $t\in\R$ we may consider the series
\begin{equation}\label{eq:spfl-smeared2}
\sum_{n\in\Z} f_n J^t_n,  \sum_{r\in \mp t+\frac12+\Z} f_r G^{\pm,t}_r, \quad
\sum_{n\in\Z} f_n L^t_n.
\end{equation}
For any $f\in \Cci_c(\S \setminus \{-1\})$ the Fourier coefficients 
$f_r= \frac{1}{2\pi}\int_{-\pi}^{\pi} \rme^{-\rmi r\theta} f(\rme^{\rmi \theta})\rmd \theta$, with $r\in \mp t+\frac12+\Z$ are rapidly decreasing for each fixed $t\in \R$. Hence, as in Sect. \ref{sec:SVir-net}, it follows from the energy bounds in Eq. \eqref{eq:LEB1} that the series define
closable operators on $\Cci(L_0)$, and we denote their closures by $J^t(f), G^{\pm,t}(f),L^t(f)$. In the light of the analogous definition with $t=0$, we also define the selfadjoint fields
\begin{equation}\label{eq:spfl-fieldsum}
G^{i,t}(f) := \frac{ \Big( ( G^{+,t}(f) - (-1)^{i}G^{-,t}(f))|_{\Cci(L_0)}\Big)^-}{\sqrt{2}} , \quad i=1,2 .
\end{equation}
As in the case $t=0$ it can be shown that $J^t(f), L^t(f), G^{i,t}(f)$, $i=1,2$ are selfadjoint for all $f\in \Cci_c(\S \setminus \{-1\}, \R)$.

We shall see that the action of $\eta_t$ on the even generators corresponds, through $\alpha$-induction, to a $\Uone$-automorphism $\rho_q$ \cite[Sect. 2]{BMT} with charge $q=tc/3$ of the subnet  $\A_{\Uone}  \subset  \A_c$ generated by the current 
$J(z)= \sum_{n \in \Z} J_n z^{-n-1} $.

\begin{theorem}\label{th:spfl1}
For every $t\in\R$, there is a 
$\PSL$-covariant soliton $\bar{\eta}_t$ of $\A_c$ such that 
$t\mapsto \bar{\eta}_{t,I}$, $I\in\I_\R$, is a one-parameter group of automorphisms of $\A_c(I)$ satisfying 
\[
\bar{\eta}_{t,I}(\rme^{\rmi X (f)}) = \rme^{\rmi X^t(f)}, 
\quad f\in \Cci_c(I,\R),
\; X=J,\, G^1,\, G^2, \, L.
\]
 $\bar{\eta}_t$ is  unitarily equivalent 
to the $\alpha^\pm$-induction of a localized $\Uone$-current automorphism
$\rho_q$ with charge $q=\frac{c}{3}t$. For $t\in \Z$,  $\bar{\eta}_t$ is
a Neveu-Schwarz representation of the graded-local net $\A_c$, while for $t\in\frac12 +\Z$, $\bar{\eta}_t$ it is a Ramond representation of $\A_c$.
\end{theorem}

The proof proceeds in several steps. First, let us define representations for each local algebra. 
Let $I\in\I_\R$ and let $\phi_{I}\in\Cci_c(\S\setminus\{-1\})$ be a real valued function such that $\phi_{I}|_I = -{\rmi} \log $, where $\log$ is determined by $\log(1)=0$. Set
\begin{equation}\label{eq:spfl1}
 \bar{\eta}_{t,I}(x) := \Ad(\rme^{\rmi t J(\phi_{I})})(x), \quad x\in\A_c(I).
\end{equation}

\begin{lemma}\label{lem:spfl3}
Let $I\in \I_\R$. Then, for all $t\in \R$ and all $f\in\Cci_c(I,\R)$ we have $$ \rme^{\rmi t J(\phi_{I})} X(f) \rme^{-\rmi t J(\phi_{I})}= X^t(f),$$ 
$X=J,\, G^1,\, G^2, \, L$. Moreover, 
$\bar{\eta}_{t,I}(\A_c(I))=\A_c(I)$ and hence the map $t \mapsto \bar{\eta}_{t,I}$ is a one-paramter group of automorphisms of the 
von Neumann algebra $\A_c(I)$. If $I_1 \in \I_\R$ contains $I$ then $\bar{\eta}_{t,I_1}|_{\A_c(I)}= \bar{\eta}_{t,I}$.
\end{lemma}

\begin{proof}
Let $I\in \I_R$ and let $f\in\Cci_c(I,\R)$. Then by (``smearing'') the relations in Definition \ref{def:spfl-alg}, we have
\[
J^t(f) \psi = \left(J(f)  +  \frac{c}{3}t \int_\S  f \right) \psi,
\]
\[
L^t(f) \psi= \left( L(f) + \left( t\int_\S f \right) J(f) + (t\int_\S f)^2 \right) \psi,
\]
and
\[
G^{t,\pm}(f) \psi = \sum_{r\in\frac12\mp t+\Z} f_r G_r^{t,\pm} \psi =
\sum_{r\in\frac12 + \Z} f_{r\mp t} G_r^\pm \psi= G^\pm ( \rme^{\pm \rmi t \phi_I}  f) \psi,
\]
for all $\psi \in \Cci(L_0)$.

Now, let $\H_1$ denote the Banach space obtained by endowing the domain of $L_0$ with the norm $\|\psi \|_1:= \|(L_0 + 1) \psi\|$. 
It contains $\Cci(L_0)$ as a dense subspace. 

As a consequence of the energy bounds in Eq. \eqref{eq:LEB1}, the selfadjoint operators 
$X^t(f)$,  $X=J,\, G^1,\, G^2, \, L$, give rise to bounded linear map $B_X(t) : \H_1 \to \H$ for all $t\in \R$. Moreover, from the energy bounds, the NS $N=2$ super-Virasoro (anti-) commutation relations and the equalities above it is straightforward to see that 
$t\mapsto B_X(t)\in B(\H_1, \H)$ is norm-differentiable and that the derivative $\frac{\rmd}{\rmd t}B_X(t)$ satisfies 
$$ \left( \frac{\rmd}{\rmd t}B_X(t) \right) \psi = i[J(\phi_I), X^t(f)]\psi$$
for all $\psi \in \Cci(L_0)$.

By \cite[Lemma 4.6]{CLTW}, cf. also (the proof of) \cite[Prop. 2.1]{Tol}, we have $\rme^{\rmi t J(\phi_I)}\Cci(L_0) = \Cci(L_0)$ for all 
$t\in \R$. Let $\psi_0\in\Cci(L_0)$. Then $\psi(t) := \rme^{\rmi t J(\phi_I)} \psi_0 \in \Cci(L_0)$, for all $t\in \R$ and it follows from the proof of 
\cite[Cor. 2.2]{Tol} (cf. also the proof of \cite[Lemma 4.6]{CLTW}) that the map 
$\R \to \H_1$ given by $t \mapsto \psi(t)$ is norm differentiable with derivative $\frac{\rmd}{\rmd t}\psi(t) = \rmi J(\phi_I)\psi(t)$.  

Now, let $\psi_X(t):= X^t(f) \rme^{\rmi t J(\phi_I)} \psi_0 = B_X(t)\psi(t)$, $t\in\R$. By the above discussion we have $\psi_X(t) \in \Cci(L_0)$ for all $t\in \R$. Moreover, the map $t \mapsto \psi_X(t) \in \H$ is norm differentiable with derivative 

\begin{eqnarray*}
\frac{\rmd}{\rmd t}\psi_X(t) &=& \left(\frac{\rmd}{\rmd t} B_X(t) \right) \psi(t) + B_X(t) \frac{\rmd}{\rmd t}\psi(t)\\
&=& \rmi[J(\phi_I), X^t(f)]\psi(t) + \rmi X^t(f)J(\phi_I)\psi(t) \\
&= & \rmi J(\phi_I)X^t(f)\psi(t) \\
&= & \rmi J(\phi_I)\psi_X(t).
\end{eqnarray*}

The unique solution $\psi_X(t)$ of the above abstract Schr\"{o}dinger equation with initial value $\psi_X(0)= X(f)\psi_0$ is given by 
$\psi_X(t) =  \rme^{\rmi t J(\phi_I)} X(f)\psi_0$. Therefore, since $\psi_0 \in \Cci(L_0)$ was arbitrary we find  
$ X^t(f) \rme^{\rmi t J(\phi_I)} \psi   = \rme^{\rmi t J(\phi_I)} X(f) \psi $ for all $\psi \in \Cci(L_0)$ and hence, recalling that 
$\rme^{\rmi t J(\phi_I)}\Cci(L_0) = \Cci(L_0)$, we can conclude that  $ X^t(f) \psi   = \rme^{\rmi t J(\phi_I)} X(f) \rme^{- \rmi t J(\phi_I)} \psi $
for all $\psi \in \Cci(L_0)$. Then, the desired equality of selfadjoint operators  $ X^t(f)  = \rme^{\rmi t J(\phi_I)} X(f) \rme^{- \rmi t J(\phi_I)}$
follows from the fact that $\Cci(L_0)$ is a core for $X(f)$ and $X^t(f)$. 

As a consequence we have $\bar{\eta}_{t,I}(\rme^{\rmi X(f)})=  \rme^{\rmi X^t(f)}$, $X=J,\, G^1,\, G^2, \, L$, for all $f \in \Cci_c(I,\R)$. 
It follows that $\bar{\eta}_{t,I}$ does not depend on the choice of $\phi_I$. In particular for any given interval 
$\tilde{I} \in \I_\R$ containing the closure of $I$ we can choose $\phi_I$ with support contained in $\tilde{I}$ so that 
$\bar{\eta}_{t,I}(\A_c(I))= \rme^{\rmi t J(\phi_I)} \A_c(I) \rme^{-\rmi t J(\phi_I)} \subset \A_c(\tilde{I})$ and since 
$\tilde{I} \supset \bar{I}$ was arbitrary we can conclude that $\bar{\eta}_{t,I}(\A_c(I)) \subset \A_c(I)$ and hence that 
$\bar{\eta}_{t,I}(\A_c(I)) = \A_c(I)$. 
Now, if $I_1 \in \I_\R$ contains $I$ then the equality $\bar{\eta}_{t,I_1}|_{\A_c(I)}= \bar{\eta}_{t,I}$ easily follows from the definition and from graded-locality of the net $\A_c$.
\end{proof}
\begin{lemma}\label{lem:spfl1}
The family $(\bar{\eta}_{t,I})_{I\in\I_\R}$ forms a (locally normal) $\PSL$-covariant soliton of the graded-local conformal net $\A_c$.
\end{lemma}

\begin{proof}
The normality of $\bar{\eta}_{t,I}$, $I\in \I_\R$, is obvious from the definition and the compatibility with isotony of the family 
$(\bar{\eta}_{t,I})_{I\in\I_\R}$ follows from Lemma \ref{lem:spfl3}.
Here we have to establish the covariance.  Let $U_q$ be the unitary representation obtained by integrating to $\PSI$ the representation of the skew-adjoint part of the complex Lie algebra generated by $L^t_{-1},L^t_0,L^t_1$. We have to show that, for all $I \in \I_\R$, 
with $q=\frac{c}{3}t$ as before, 

\begin{equation}\label{etacov}
 U_q(g) \bar{\eta}_{t,I}(x) U_q(g)^* = \bar{\eta}_{t,\dot{g}I}(U(g)xU(g)^*), \quad x\in\A_c(I),
\end{equation}
for $g\in \mathcal{U}_I$, where $\mathcal{U}_I$ is the connected component of the identity in $\PSI$ of the open set 
$\{g \in \PSI: \dot{g}I \in \I_\R \}$.

Now, let $I \in \I_\R$ and let $f\in \Cci_c(I,\R)$. Arguing then as in the proof of Theorem \ref{th:SVir2-diff-cov} we obtain 

\[
U_q(g) X^t(f) U_q(g)^* = X^t(\beta_{d(X)}(\dot{g})f), \quad g\in \mathcal{U}_I. 
\]

Here, for $d \geq 0$ and $g\in \PSL$, $\beta_d(g)f$ is the function on $\S$ defined by   
$$ \left( \beta_d(g)f \right)(\rme^{\rmi\theta}) :=  \left( -\rmi \frac{\rmd}{\rmd \theta}\log (g^{-1} \rme^{\rmi\theta})   \right)^{1-d} 
f(g^{-1} \rme^{\rmi\theta})
$$
and the subscript $d(X)$ is determined by $d(J)=1$, $d(G^i)=3/2$, $d(L)=2$. 

Appealing to Lemma \ref{lem:spfl3}, we find
\begin{align*}
U_q(g) \bar{\eta}_{t,I} \left( \rme^{\rmi X(f)} \right) U_q(g)^*
=& U_q(g) \Ad(\rme^{\rmi J(\phi_{t,I})})(\rme^{\rmi X(f)})) U_q(g)^*\\
=& U_q(g) \rme^{\rmi X^t(f)} U_q(g)^* \\
=& \rme^{\rmi (X^t(\beta_{d(X)}(g).f)}\\
=& \Ad(\rme^{\rmi J(\phi_{t,\dot{g}I})})(\rme^{\rmi X(\beta_{d(X)}(g).f)}))\\
=& \bar{\eta}_{t,\dot{g}I}\left( \rme^{\rmi X(\beta_{d(X)}(g).f)}) \right)\\
=& \bar{\eta}_{t,\dot{g}I} \left( U(g) \rme^{\rmi X(f)} U(g)^* \right), \quad g \in \mathcal{U}_I.
\end{align*}
Thus, since the unitaries $ \rme^{\rmi X(f)}$, with $f \in \Cci_c(I,\R)$ generate  $\A_c(I)$, Eq. \eqref{etacov} now follows from the normality 
of $\bar{\eta}_{t,\dot{g}I}$ for all $g \in \mathcal{U}_I$ and since $I \in \I_\R$ was arbitrary, we can conclude that $\bar{\eta}_t$ is a 
$\PSL$-covariant soliton in the sense of Definition \ref{def:CFT-covering-reps}, \cf Remark \ref{rem:CFT-IR}.
\end{proof}

\begin{lemma}\label{lem:spfl4}
Given $t\in\R$, for every $I_0\in\I_\R$, there is a unitary $u_{t,I_0}\in B(\H)$ such that
\[
\bar{\eta}_{t,I} = \Ad(u_{t,I_0})\circ \alpha_{\rho^{I_0}_q,I}^+
=\Ad(\rme^{\rmi 2\pi t J_0} u_{t,I_0})\circ \alpha_{\rho^{I_0}_q,I}^-, \quad I\in\I_\R,
\]
where $\rho_q^{I_0}$ is an endomorphism of the $\Uone$-subnet of $\A_c$, with charge $q=\frac{c}{3}t$ and localized in $I_0$. 
\end{lemma}

\begin{proof}
The global endomorphism of $\rho_q$ of the $\Uone$-subnet $\A_{\Uone}\subset\A_c$ is defined by
\[
\rho_q(\rme^{\rmi J(f)}) = \rme^{\rmi (J(f)+q\int f)}.
\]
Given $I_0\in\I_\R$, fix a smooth $2\pi$-periodic function $h_{I_0}:\R\ra\R$ which, restricted to $(-\pi,\pi)$, satisfies
\[
h_{I_0}( \theta)= \left\lbrace \begin{array}{l@{\; :\; }l}
\theta & \theta < I_0\\
\textrm{arbitrary} & \theta\in I_0\\
\theta - 2\pi & \theta >I_0.
\end{array}
\right.
\]
Then $\rho^{I_0}_q:= \Ad(\rme^{-t \rmi J(h_{I_0})})\circ\rho_q$ is an endomorphism of $\A_{\Uone}$ localized in ${I_0}$ and equivalent to $\rho_q$. Recall \cite{LR} that the $\alpha$-induced sectors of $\rho_q^{I_0}$ may be expressed as
\[
\alpha^\pm_{\rho_q^{I_0},I} = \Ad(z(\rho_q^{I_0},g_\pm)) = \Ad(\rme^{t \rmi J(h_{g_\pm I_0}-h_{I_0})}),
\]
where $z(\rho_q^{I_0}, g )$ is the cocycle associated to he covariant representation $\rho_q^{I_0}$, see \eg \cite{CHL,GL}, and 
$g_\pm \in G$ are such that $g_- I_0 < I < g_+ I_0$.

For $x\in\A_c(I)$ we now obtain
\[
\Ad(\rme^{t\rmi J(h_{I_0})})\circ \alpha^\pm_{\rho_q^{I_0},I}(x)
=\Ad(\rme^{t\rmi J(h_{I_0})} \rme^{t\rmi J(h_{g_\pm I_0}-h_{I_0})})(x)
=\Ad(\rme^{t\rmi J(h_{g_\pm I_0})})(x) .
\]
Since $g_-I_0 < I < g_+ I_0$, we see that $h_{g_+ I_0}|_{I}=\iota|_{I}$, while $h_{g_- I_0}|_{I}=\iota|_{I} -2\pi$. Thus the definition of $\phi_{t,I}$ and $\bar{\eta}_t$ finally implies
\[
\Ad(\rme^{t\rmi J(h_{I_0})})\circ \alpha^\pm_{\rho_q^{I_0},I} = \left\lbrace \begin{array}{l@{\; :\; }l}
\bar{\eta}_{t,I} & ``+ "\\
\Ad(\rme^{-\rmi 2\pi t J_0})\circ \bar{\eta}_{t,I} & ``- ",
\end{array}
\right.
\]
so we are done, setting $u_{t,I_0}:= \Ad(\rme^{t\rmi J(h_{I_0})})$, independent of $I$. 
In particular, the $\alpha^\pm$-induced solitons differ by the gauge automorphism $\Ad(\rme^{\rmi 2\pi t J_0})$ of the net $\A_c$, which is trivial if $t\in \Z$. 
\end{proof}

\noindent\textbf{Proof of Theorem \ref{th:spfl1}.} 
By Lemma \ref{lem:spfl3}, $\bar{\eta}_t$ has the desired action on the generators of the $N=2$ super-Virasoro net $\A_c$ and acts locally as an automorphism group (a ``flow''). 
Lemma \ref{lem:spfl1} tells us that $\bar{\eta}_t$ is a $\PSL$-covariant soliton of $\A_c$ which, by Lemma 
\ref{lem:spfl4}, is unitarily equivalent to the $\alpha^\pm$-induction of a localized $\Uone$-current automorphism $\rho_q$ with charge 
$q=\frac{c}{3}t$.

For all $r$ in the allowed set (depending on $t$ and on the field $X=J,G^\pm,L$) we have $[L^t_0, X^t_r]=-rX^t_r$ on 
$\Cci(L_0)$. It follows that
\begin{equation}\label{eq:spfl2pi}
\rme^{\rmi s L^t_0} X^t_r \rme^{-\rmi s L^t_0} =  \rme^{-\rmi r s} X^t_r.
\end{equation}
For $s=2\pi$ and for $X^t=J^t,L^t$, we have $r\in\Z$, so the phase factor is $\rme^{-2\pi\rmi r}=1$. For $X^t=G^{t,\pm}$ instead, we have 
$r\in\mp t + \frac12+\Z$, so  $\rme^{-2\pi\rmi r}=-\rme^{\pm 2\pi\rmi t }$. So $\rme^{2\pi\rmi L^t_0} $ implements the grading (is a scalar) precisely when $t\in\Z$ ($t\in\frac12+\Z$,\; resp.~).

Thus Lemma \ref{lem:CFT-DHR} implies that $\bar{\eta}_t$ is a $\PSL$-covariant general soliton iff $t\in\frac12\Z$, and it is either a Ramond or a Neveu-Schwarz $\PSL$-covariant general soliton in the sense of Definition \ref{def:CFT-covering-reps} and Theorem \ref{th:CFT-R-NS}, depending on whether $t\in \frac12+\Z$ or $t\in\Z$, resp.~. 
\boxy

Now let $\pi$ be an irreducible Neveu-Schwarz representation of $\A_c$ and let $\pi$ denote also the corresponding irreducible unitary representation of the Neveu-Schwarz $N=2$ super-Virasoro algebra given by Theorem \ref{th:SVir-pi-exp-pi}. Then, for $t\in \Z$ 
(resp.~ $t\in \frac12 +\Z$), $\pi\circ \eta_t$, is an irreducible unitary  representation of the Neveu-Schwarz (resp.~ Ramond) $N=2$ 
super-Virasoro algebra on $\H_\pi$. It follows from Theorems \ref{th:SVir-pi-exp-pi} and \ref{th:spfl1} and the Trotter product 
formula that for all $I \in \I_\R$ we have
\[
\pi_I\circ \bar{\eta}_{t,I}(\rme^{\rmi X (f)}) = \rme^{\rmi X^{\pi\circ \eta_t}(f)}, 
\quad f\in
\Cci_c(I,\R),
\; X=J,\, G^1,\, G^2, \, L.
\]

It follows that the family $\pi\circ \bar{\eta}_{t} := (\pi_I\circ \bar{\eta}_{t,I})_{I \in \I_\R}$ defines an irreducible Neveu-Schwarz (resp. Ramond) representation of $\A_c$. If $\pi$ is a Ramond representation we have a similar situation. We record these facts in the following proposition. 

\begin{proposition}\label{spflowrep} If $\pi$ is an irreducible Neveu-Schwarz representation of $\A_c$ then $\pi\circ \bar{\eta}_{t}$
is an irreducible Neveu-Schwarz (resp. Ramond) representation of $\A_c$ for all $t \in \Z$ (resp.~ $t \in \frac12 + \Z$). 
If $\pi$ is an irreducible Ramond representation of $\A_c$ then $\pi\circ \bar{\eta}_{t}$ is an irreducible Neveu-Schwarz (resp.~ Ramond) representation of $\A_c$ for all $t \in \frac12 + \Z$ (resp.~ $t \in \Z$). In particular $\bar{\eta}_{1/2}$ gives rise to a one-to one correspondence between the irreducible Neveu-Schwarz and Ramond representations of $\A_c$
\end{proposition}

\section{The Coset Identification for the $N=2$ Super-Virasoro Nets with $c<3$}\label{sec:coset}

In this section we prove a crucial result for our analysis 
(Theorem \ref{th:CosetId}).
As explained in the introduction, claims of this result have appeared in the literature in the vertex algebraic context, but we could not find any satisfactory and complete proof. The proof is operator algebraic in nature but it can be shown that it covers the original vertex algebraic statement as a consequence of the close relationship between the two approaches.

As usual, we denote by $\A_{\SUtwo_n}$ the completely rational local conformal net associated to the level $n$ positive energy representations of the loop group $L\SUtwo$, \cf \cite{FG,Was}. Similarly we denote by $\A_{\Uone_{2n}}$ the  local conformal net corresponding to 
the positive energy representations of the loop group $L\Uone$ at level $2n$. For every positive integer $n$, $\A_{\Uone_{2n}}$ is completely rational with $2n$ sectors, all with statistical dimension one, \cf \cite{X-3mf,Xu-Str-Add}. In fact the nets $\A_{\Uone_{2n}}$
coincide with the local extensions of the $c=1$ net $\A_{\Uone}$ generated by a chiral $\Uone$ current first classified in \cite{BMT}, \cf also \cite{Xu-Str-Add}. The inclusion $\Uone \subset \SUtwo$ gives rise to the diagonal inclusion $\Uone \subset \SUtwo \times \Uone$ which corresponds to the inclusions of local conformal nets  $\A_{\Uone_{2n+4}} \subset \A_{\SUtwo_n} \otimes \A_{\Uone_{ 4}}$ and one can consider the corresponding coset nets, see e.g. \cite{X-coset,X-3mf}. We begin our analysis of the $N=2$ coset identification with the following theorem.

\begin{theorem}\label{th:bosonic-coset}
The coset 
net corresponding to the inclusion
 $\A_{\Uone_{2n+4}} \subset \A_{\SUtwo_n} \otimes \A_{\Uone_{ 4}}$ is completely rational
 for every positive integer $n$.
 Its list of irreducible representations is numbered by the following $(l,m,s)$ satisfying
$l=0,1,2,\dots,n$,
$m=0,1,2,\dots,2n+3 \in{\mathbb{Z}}/(2n+4){\mathbb{Z}}$,
$s=0,1,2,3\in{\mathbb{Z}}/4{\mathbb{Z}}$
with $l-m+s\in 2{\mathbb{Z}}$
with the identification $(l,m,s)=(n-l,m+n+2,s+2)$.
                                             
The fusion rules are given as follows, treating the three components $l,m,s$ in the label $(l,m,s)$ separately: For the first component $l$, we use the usual
${\mathrm{SU}}(2)_n$ fusion rules.  For the second component $m$,
we use the group multiplication in
${\mathbb{Z}}/(2n+4){\mathbb{Z}}$. For the third component
$s$, we use the group multiplication in
${\mathbb{Z}}/4{\mathbb{Z}}$.
All these products are with the identification $(l,m,s)=(n-l,m+n+2,s+2)$.

The univalence (statistics phase) $\rme^{\rmi 2\pi L_0}$
and dimension of the irreducible DHR sector $(l,m,s)$ are given by
$$\exp\left(\left(\frac{l(l+2)-m^2}{4(n+2)}+\frac{s^2}{8}\right)
2\pi i\right), \quad \sin((l+1)\pi/(n+2))/\sin(\pi/(n+2)).$$
Accordingly the statistical dimension is $1$ (i.e. we have automorphisms)
iff either $l=0$ or $l=n$.
\end{theorem}

\begin{proof}
This coset net is a special case of coset net studied in \cite{X-coset}. In the notation of \cite{X-coset}, this coset net is $\A(G(1,1,n)).$ By (1) of Theorem 2.4 in \cite{X-coset}, $\A(G(1,1,n))$ is completely rational. By Theorem 4.4 in \cite{X-coset}, the Vacuum Pairs in this case is an order two abelian group  generated $(n,n+2,2).$ 

Note that this group acts without fixed points on the $(l,m,s)$ as given above. By Theorem 4.7 in \cite{X-coset}, $(l,m,s)$ are irreducible representations of $\A(G(1,1,n)).$
The rest of the statement in the theorem follows by the remark after Theorem 4.7 in \cite{X-coset}.
\end{proof}

As already mentioned in Theorem \ref{th:SVir2-reps} the unitary representations of the super-Virasoro algebra with central charge $c_n=3n/(n+2)$ have been explicitly realized by Di Vecchia, Petersen, Yu and Zheng, using the coset construction for the inclusion
$\Uone_{2n+4}\subset \SUtwo_n \otimes \CAR^{\otimes 2}$ \cite{DVYZ}.
Now let $\A_{\Uone_{2n+4}} \subset \A_{\SUtwo_n} \otimes \A_{\CAR^{\otimes 2}}$ be the corresponding inclusion
of conformal nets, where $\A_{\CAR^{\otimes 2}}$ is the net generated by two real chiral free Fermi fields (equivalently one complex chiral free Fermi field), see e.g. \cite{Boc,Boc2}, and let $\mC_n$ be the corresponding coset net defined by
\begin{equation}
\mC_n(I) = \A_{\Uone_{2n+4}}(\S)' \cap \big( \A_{\SUtwo_n}(I) \otimes \A_{\CAR^{\otimes 2}}(I) \big), \quad I\in \I ,
\end{equation}
where $\A_{\Uone_{2n+4}}(\S) := \bigvee_{I\in \I} \A_{\Uone_{2n+4}}(I)$.

Using the fact that the even part of the NS representation space of $\CAR^{\otimes 2}$ carries the vacuum representation of $\Uone_4$ 
(see e.g. \cite[Sect. 5B]{BMT}) one can conclude that the even part $\mC_n^\gamma$ of the Fermi conformal net
$\mC_n$ is given by the coset
\begin{equation}
\mC_n^\gamma(I) = \A_{\Uone_{2n+4}}(\S)' \cap  \big(  \A_{\SUtwo_n}(I) \otimes \A_{\Uone_4}(I) \big) , \quad I\in \I.
\end{equation}
In analogy with the cases $N=0$ \cite{KL1}, and $N=1$ \cite{CKL} one can show that the results in \cite{DVYZ}
imply that the $N=2$ super Virasoro net $\A_{c_n}$ is a covariant irreducible subnet of the coset net $\mC_n$.
The aim of this section is to prove that these nets actually coincide, \ie that  $\A_{c_n}=\mC_n$ (the $N=2$
coset identification).

Let us denote by $\pi_m$, $m \in \Z/(2n+4)\Z$ the irreducible representations of the net  $\A_{\Uone_{2n+4}}$
by $\pi_l$, $l=0, 1, \dots, n$ the irreducible representations of   $\A_{\SUtwo_n}$ and by $\pi_{NS}$ the vacuum
(Neveu-Schwarz) representation of  $\A_{\CAR^{\otimes 2}}$.
Then the inclusion $\A_{\Uone_{2n+4}}\otimes \mC_n \subset \A_{\SUtwo_n} \otimes \A_{\CAR^{\otimes 2}}$
gives decompositions
\begin{equation}
\pi_l \otimes \pi_{NS}|_{\A_{\Uone_{2n+4}} \otimes \; \mC_n}=\bigoplus_{m \in \Z/(2n+4)\Z} \pi_m \otimes \pi_{(l,m)},
\end{equation}
where $\pi_{(l,m)}$ is the (possibly zero) NS representation of $\mC_n$ on the multiplicity space of $\pi_m$. The corresponding
Hilbert spaces $\H_{l,m}$ carries unitary representations of the NS $N=2$ super-Virasoro algebra with central charge
$c_n$. Now let \begin{equation}
\chi_{(l,m)}(t) = \tr_{\H_{(l,m)}} t^{L^{\pi_{(l,m)}}_0}
\end{equation}
be the character of $\pi_{(l,m)}$ (branching function). Although not explicitly stated there the following proposition follows directly from the construction in \cite{DVYZ}.

\begin{proposition}\label{prop:charcoset}
If $|m| \leq l $ and $l + m \in 2\Z$ then the unitary representation of the NS $N=2$ super-Virasoro algebra on $\H_{l,m}$ with central charge $c_n$ contains a subrepresentation with $h= h_{l,m}:= \frac{l(l+2)-m^2}{4(n+2)}$ and 
$q=q_{n,m}:= -\frac{m}{n+2}$. Moreover, $\chi_{(l,m)}(t) = t^{h_{l,m}} + o(t^{h_{l,m}})$ as $t \to 0^+$, namely $h_{l,m}$ is the lowest conformal energy eigenvalue on $\H_{l,m}$ and the corresponding multiplicity is one.
\end{proposition}

The following lemma will play a crucial role in the proof of the $N=2$ coset identification.
\begin{lemma}\label{lemma:irred-restr}
If $n$ is even (resp. odd) then the restriction of $\pi_{(n,0)}$ (resp. $\pi_{(n, \pm1)}$) to $\A_{c_n}$ is irreducible.
\end{lemma}
\begin{proof}
Let $n$ be even. By Proposition \ref{prop:charcoset} the Hilbert space $\H_{(n,0)}$  is a direct sum
$\K_1\oplus \K_2$ where $\K_1$ carries an irreducible representation of the NS $N=2$ super-Virasoro
algebra $\SVir^{N=2}$ with central charge $c_n$ and lowest energy $h_{(n,0)} = n/4$ and $\K_2$ is either zero or carries a unitary representation of $\SVir^{N=2,0}$ with central charge $c_n$ and lowest energy $h > n/4$.
But $n/4$ is the maximal possible value for the lowest energy and hence $\K_2={0}$. The case $n$ odd is similar.
\end{proof}

We are interested in the (NS) DHR sectors of $\mC_n$. They are labeled with $(l,m)$ satisfying
$l=0,1,2,\dots,n$,
$m=0,1,2,\dots,2n+3 \in{\mathbb{Z}}/(2n+4){\mathbb{Z}}$,
with $l-m \in 2{\mathbb{Z}}$ with the identification $(l,m)=(n-l,m+n+2)$. The restriction
of $(l,m)$ to $\mC_n^\gamma$ is given by $(l,m,0) \oplus (l,m,2)$. Moreover $(l,m)=\alpha_{(l,m,0)}$
where $\alpha_{(l,m,0)}$ denotes the $\alpha$-induction of $(l,m,0)$ from $\mC_n^\gamma$ to $\mC_n$.
In view of Theorem \ref{th:bosonic-coset} then we have the following fermionic fusion rules:
\begin{equation}\label{eq:FermiFusions}
(l_1,m_1)(l_2,m_2)= \bigoplus_{ \begin{array}{c} | l_1-l_2 | \leq l \leq \min \{l_1+l_2, 2n-l_1-l_2\} \\ l+l_1+l_2 \in 2\Z \end{array}} (l, m_1+m_2).
\end{equation}
Automorphisms correspond to $l=0, n$. It follows from Theorem \ref{th:bosonic-coset} and its proof that   
$[\pi_{(l,m)}] = (l,m)$ for $l-m \in 2{\mathbb Z}$. 
In particular $[\pi_{(n,0)}] = (n, 0)$ for $n$ even and $[\pi_{(n,\pm 1)}] = (n, \pm 1)$ for $n$ odd where 
$\pi_{(n,0)}$ and $\pi_{(n,\pm 1)}$ are the representations in Lemma \ref{lemma:irred-restr}. Accordingly, for $n$ even, $(n, 0)$ remains irreducible when restricted to $\A_{c_n}$ and similarly,  for $n$ odd, $(n, \pm 1)$ remain irreducible when restricted to $\A_{c_n}$.

Now let $\A_c$ be the $N=2$ super-Virasoro net with central charge $c$ (not necessarily $c<3$) and let
$J(z)=\sum_{n\in \Z}J_nz^{-n-1}$ be the corresponding current with Fourier coefficients satisfying the commutation relations
\begin{equation}\label{eq:Jn}
\left[J_n,J_m \right] = \frac{c}{3}n \delta_{n+m,0}.
\end{equation}
The current $J(z)$ generates a subnet $\A_{\Uone} \subset \A_c$ isomorphic to the $\Uone$ net in  \cite{BMT}. We can label the sectors of $\A_{\Uone}$ by $(q)$, $q\in \R$, corresponding to $J(z) \mapsto J(z) + q z^{-1}$. They satisfy the DHR fusions $(q_1)(q_2)=(q_1+ q_2)$ see \cite{BMT} and \cite{Xu-Str-Add}.
Fix an interval $I_0 \in \I_\R$. For every $q\in \R$ we choose an endomorphism $\rho_q$ of $\A_{\Uone}$,
localized in $I_0$ and such that $[\rho_q]=(q)$.
Now let $\B$ be a diffeomorphism covariant graded local extension of $\A_{c}$.
Following \cite{LR} we shall denote by $\pi^0$ the vacuum representation of $\B$, by $\pi_0$ the vacuum representation of $\A_{c_n}$ and by $\pi =(\pi^{0})^{rest}$ the restriction of $\pi^0$ to $\A_{c_n}$. Now, having the inclusions $\A_{\Uone} \subset \A_{c} \subset \B$ we can consider
the $\alpha$-inductions (say $\alpha^+$) $\alpha^{\A_{c}}_{\rho_q}$ and  $\alpha^{\B}_{\rho_q}$ of $\rho_q$
to $\A_c$ and $\B$ resp.~. Note that the restriction to $\A_c$ of $\pi^0 \circ \alpha^\B_{\rho_q}$ is
$\pi \circ \alpha^{\A_c}_{\rho_q}$. Note also that by Theorem \ref{th:spfl1} $\alpha^{\A_{c}}_{\rho_{\frac{c}{3}t}}$ is
unitarily equivalent to our operator algebraic version $\bar{\eta}_t$ of the $N=2$ spectral flow. Accordingly
$\alpha^{\B}_{\rho_{\frac{c}{3}t}}$ is a natural candidate to represent the unitary equivalence class of a possible extension of the \emph{spectral flow} on $\B$.

Now let $\H_\B$ be the vacuum Hilbert space of $\B$ and $\H_{\A_c}$ be the vacuum Hilbert space of $\A_c$. Since
$\rme^{\rmi \pi J_0}$ is the grading unitary on $\H_{\A_c}$, we have $\rme^{ i 2\pi J_0}=1$ on $\H_{\A_c}$ and accordingly
the spectrum of $J_0$ on $\H_{\A_c}$ is contained in $\Z$. In fact it is not hard to see that this spectrum is exactly 
$\Z$. However the spectrum of $J_0$ on $\H_\B$ can be in general larger than $\Z$ even when $\B$ is an irreducible extension of $\A_c$. This is however not the case  if e.g. $\rme^{\rmi  \pi J_0}$ is still the grading unitary on $\B$, a condition which may be seen as a regularity condition on the extension $\B$.

\begin{theorem}\label{th:BMTinduction} 
Assume that the spectrum of $J_0$ on $\H_\B$ is $\Z$. Then, for any $t\in \Z$,
$\pi^0 \circ \alpha^{\B}_{\rho_{\frac{c}{3}t}}$ is a NS representation of $\B$. In particular it restricts to a DHR representation of the even subnet $\B^\gamma$ of $\B$.
\end{theorem}

\begin{proof} Let $\Gamma$ be the grading unitary on $\H_\B$. We have the spin statistic relation
$\Gamma = \rme^{ i 2\pi L_0}$. Moreover $\Gamma$ commutes with the $\alpha$-induction, namely,
for any $I \in \I_\R$,
$$\Gamma\pi^0 \circ \alpha^{\B}_{\rho_{\frac{c}{3}t}}(b)\Gamma= \pi^0 \circ \alpha^{\B}_{\rho_{\frac{c}{3}t}}
(\Gamma b \Gamma)$$
for all $b \in \B(I)$. To see this let us recall that if $\tilde{I} \in \I_\R$ is sufficiently large then there is
a unitary $u \in \A_{\Uone}(\tilde{I})$ such that $ubu^* = \alpha^{\B}_{\rho_{\frac{c}{3}t}}(b)$, for all
$b\in \B(I)$. and the claim follows from the fact that $u$ commutes with $\Gamma$.
Now let $\D$ be the covariant subnet of $\B$ defined by
$$\D(I)=\A_{\Uone}(\S)'\cap \B(I), \quad I\in \I.$$
Note that $\D$ is trivial iff the central charge of the net $\B$ is $1$.
Then  we have the inclusion $\A_{\Uone} \otimes \D \subset \B$. Moreover the subnet $\A_{\Uone} \otimes \D$
contains the Virasoro subnet of $\B$. The fact that the spectrum of $J_0$ on $\H_\B$ is $\Z$ implies that the restriction
of $\pi^0$ to $\A_{\Uone}$ is unitarily equivalent to a direct sum of representations $\rho_q$ with
$q \in \Z$. Hence the restriction of $\pi^0$ to $\A_{\Uone} \otimes \D$ can be written as a direct sum
$$\bigoplus_{q\in \Z} \rho_q\otimes \sigma_q$$
where $\sigma_q$ is the representation of $\D$ on the (possibly zero) multiplicity space of $\rho_q$.
Accordingly the conformal vacuum Hamiltonian $L_0$ has the following decomposition
$$L_0  = \bigoplus_{q\in \Z} \left( L^{\rho_q}_0\otimes 1 + 1\otimes L^{\sigma_q}_0\right).$$
Now let us denote $\pi^0 \circ \alpha^{\B}_{\rho_{\frac{c}{3}t}}$ by $\lambda_t$.
Since ${\rho_{\frac{c}{3}t}}$ is M\"{o}bius covariant, $\lambda_t$ is a  M\"{o}bius covariant
soliton of $\B$. Moreover the restriction of $\lambda_t$ to $\A_{\Uone} \otimes \D$ is
$$\bigoplus_{q\in \Z} \rho_{q + \frac{c}{3}t}\otimes \sigma_q$$ and we have
$$L^{\lambda_t}_0  = \bigoplus_{q\in \Z} \left( L^{\rho_{q + \frac{c}{3}t}}_0\otimes 1 + 1\otimes L^{\sigma_q}_0\right).$$
We want to compute the 
univalence operator $\rme^{\rmi 2\pi L^{\lambda_t}_0}$ . Recall that the lowest energy in the representation space of $\rho_q$ is given by 
$\frac{3}{c}\frac{q^2}{2}$, see e.g. \cite{BMT} (the factor $\frac{3}{c}$ is due to the factor
$\frac{c}{3}$ in the commutation relations in \eqref{eq:Jn}). It follows that
\begin{eqnarray*}
\rme^{\rmi 2\pi L^{\lambda_t}_0}  &=& \bigoplus_{q\in \Z}  \rme^{\rmi 2\pi \frac{3}{2c}(q + \frac{c}{3}t)^2}\otimes \rme^{\rmi 2\pi L^{\sigma_q}_0} = \bigoplus_{q\in \Z}  \rme^{\rmi 2\pi \frac{3}{2c}(q + \frac{c}{3}t)^2}\otimes \rme^{\rmi 2\pi L^{\sigma_q}_0} \\
&=&
\rme^{\rmi 2\pi \frac{c}{6}t^2}\left(\bigoplus_{q\in \Z}  \rme^{\rmi 2\pi \frac{3}{2c}q^2} \rme^{\rmi 2\pi qt}\otimes \rme^{\rmi 2\pi L^{\sigma_q}_0} \right) \\
&=&
\rme^{\rmi 2\pi \frac{c}{6}t^2}\left(\bigoplus_{q\in \Z}  \rme^{\rmi 2\pi L^{\rho_q}_0} \rme^{\rmi 2\pi qt}\otimes \rme^{\rmi 2\pi L^{\sigma_q}_0} \right).
\end{eqnarray*}
If $t\in \Z$ then $\rme^{\rmi 2\pi qt}=1$ for all $q\in \Z$. Hence
\begin{eqnarray*}
\rme^{\rmi 2\pi L^{\lambda_t}_0} &=&
\rme^{\rmi 2\pi \frac{c}{6}t^2}\left(\bigoplus_{q\in \Z}  \rme^{\rmi 2\pi L^{\rho_q}_0} \otimes \rme^{\rmi 2\pi L^{\sigma_q}_0} \right) \\
&=& \rme^{\rmi 2\pi \frac{c}{6}t^2}\rme^{\rmi 2\pi L_0} = \rme^{\rmi 2\pi \frac{c}{6}t^2}\Gamma
\end{eqnarray*}
for all $t\in \Z$. It follows that, for any $t \in \Z$, any $I\in \I_\R$ and $b\in \B(I)$, we have
$$ \rme^{\rmi 2\pi L^{\lambda_t}_0}\lambda_t(b) \rme^{-i2\pi L^{\lambda_t}_0}= \lambda_t (\Gamma b \Gamma) $$
and the conclusion follows from Lemma \ref{lem:CFT-DHR}
\end{proof}

As pointed out before the spectrum of $J_0$ on $\H_\B$ is in general larger than $\Z$ for an arbitrary extension
$\B$ of $\A_c$. Hence, in particular, the unitary $\rme^{\rmi \pi J_0}$ does not in general implement the grading of $\B$.
However as a consequence of the following proposition it always implements a gauge automorphism of $\B$.

\begin{proposition}\label{prop:coset-gauge}
For any $t\in \R$ the operator $\rme^{\rmi  t J_0}$ is a gauge unitary of $\B$ namely
$\rme^{\rmi  t J_0}\Omega =\Omega$ and $\rme^{\rmi  t J_0}\B(I)\rme^{-i t J_0}= \B(I)$ for all $I\in \I$.
\end{proposition}
\begin{proof} Obviously we have $\rme^{\rmi  t J_0}\Omega =\Omega$. Now let $I \in \I$ be fixed and let
$\tilde{I} \in \I$ be such that the closure of $I$ is contained in $\tilde{I}$. Choose two real smooth functions
$f_1, f_2 \in \Cci(\S,\R)$ 
such that $\supp f_1 \subset \tilde{I}$, $\supp f_2 \subset I'$ and $f_1 +f_2=1$.
Then $J_0=J(f_1) + J(f_2)$ on a common core. Hence by locality and the Weyl relations we find
$$\rme^{\rmi  t J_0}\B(I)\rme^{-i t J_0} = \rme^{\rmi  t J(f_1)}\B(I)\rme^{-i t J(f_1)} \subset \B(\tilde{I})$$
and, since $\tilde{I}$ was an arbitrary interval containing the closure of $I$, we can infer that
$$\rme^{\rmi  t J_0}\B(I)\rme^{-i t J_0} \subset \B(I), \quad t \in \R,$$
and the conclusion follows.
\end{proof}

Now recall from Sect. \ref{sec:spfl} that for every $t\in \R$ there is a M\"{o}bius covariant general soliton $\bar{\eta}_t$ 
of $\A_{c}$ corresponding, in the sense of Theorem \ref{th:spfl1} (cf. also Theorem \ref{th:SVir-pi-exp-pi}), to the representation of $\SVir^{N=2, t}$ obtained on the vacuum Hilbert space $\H_{\A_c}$ by the composition of the vacuum representation of the Neveu-Schwarz $N=2$ super-Virasoro algebra $\SVir^{N=2}$ generating the net $\A_{c}$ with spectral flow $\eta_t$.
For $t\in \Z$,  $\SVir^{N=2, t}$ coincides with $\SVir^{N=2}$ and, by Theorem \ref{th:spfl1}, $\bar{\eta}_t$
is a NS representation of $\A_c$ on $\H_{\A_c}$.

\begin{proposition}\label{prop:eta1}
The representation $\bar{\eta}_1$ of $\A_c$ corresponds to the unitary irreducible representation of $\SVir^{N=2}$ with central charge $c$ and $(h,q)= (\frac{c}{6},\frac{c}{3})$.
\end{proposition}
\begin{proof} It is enough to show that the composition of the vacuum representation with central charge $c$ of
$\SVir^{N=2}$ with $\eta_1$ is the irreducible representation of $\SVir^{N=2}$ with central charge $c$ and
$(h,q)= (\frac{c}{6},\frac{c}{3})$ on $\H_{\A_c}$.  First of all note that the irreducibility of this representation follows
from that of the vacuum representation and the invertibility of the spectral flow. Now let $\Omega \in \H_{\A_{c}}$
be the vacuum vector. Then
\begin{eqnarray*}
\eta_1(L_m)\Omega & = & L_m\Omega +\frac12 J_m\Omega =0 \\
\eta_1(J_m)\Omega & = & J_m\Omega=0,
\end{eqnarray*}
for every positive integer $m$.
Moreover,
\begin{eqnarray*}
\eta_1 (G^+_r)\Omega & = & G^+_{r+1}\Omega =0 \\
\eta_1 (G^-_r)\Omega & = & G^-_{r-1}\Omega =0,
\end{eqnarray*}
for every positive $r \in \frac12 +\Z$, where in the second equation we used the fact that
$G^-_{-\frac12}\Omega=0$. It follows that $\Omega$ is a lowest energy vector also for the representation
defined by $\eta_1$ and consequently
\begin{equation*}
h \Omega  =  \eta_1(L_0)\Omega= \frac{c}{6}\Omega, \quad
q\Omega  =  \eta_1(J_0)\Omega = \frac{c}{3}\Omega.
\end{equation*}
\end{proof}

We now come back to the inclusion $\A_{c_n} \subset \mC_n$.
\begin{lemma}\label{lemma:IntegerCharge} $\rme^{\rmi \pi J_0} = \rme^{\rmi 2\pi L_0} =\Gamma$ on the vacuum Hilbert space $\H_{\mC_n}$ of $\mC_n$.
\end{lemma}
\begin{proof} By Lemma \ref{lemma:irred-restr} there is a NS representation $\tilde{\pi}$  whose restriction to
$\A_{c_n}$ is irreducible. Let $\tilde{J}(z)= \sum_{k\in \Z} \tilde{J}_k z^{-k-1}$ be the corresponding current
on $\H_{\tilde{\pi}}$.  Then
$$\rme^{\rmi \pi\tilde{J}_0}\rme^{-i2\pi L^{\tilde{\pi}}_0} \tilde{\pi}_I(x)\rme^{\rmi 2\pi L^{\tilde{\pi}}_0}  \rme^{-i\pi\tilde{J}_0}=
\tilde{\pi}_I \left( \rme^{\rmi \pi J_0}\rme^{-i2\pi L_0}  x \rme^{-i2\pi L_0} \rme^{-i\pi J_0}\right)$$
for all $I\in \I$ and all $x\in \mC_n(I)$. In particular,
$$ \rme^{\rmi \pi\tilde{J}_0}\rme^{-i2\pi L^{\tilde{\pi}}_0} \tilde{\pi}_I(x)\rme^{\rmi 2\pi L^{\tilde{\pi}}_0}  \rme^{-i\pi\tilde{J}_0} = \tilde{\pi}_I (x),$$
for all $I\in \I$ and all $x\in \A_{c_n}(I)$ and hence, by irreducibility,  $\rme^{\rmi \pi\tilde{J}_0}\rme^{-i2\pi L^{\tilde{\pi}}_0}$
must be a multiple of the identity. It follows that
$$\tilde{\pi}_I \left( \rme^{\rmi \pi J_0} \rme^{-i2\pi L_0}x \rme^{\rmi 2\pi L_0} \rme^{-i\pi J_0}\right)= \tilde{\pi}_I (x),$$
for all $I\in \I$ and all $x\in \mC_n(I)$. Accordingly $\rme^{\rmi \pi J_0}\rme^{-i2\pi L_0}$ is also a multiple of the identity and the conclusion
follows because $\rme^{\rmi \pi J_0}\rme^{-i2\pi L_0}\Omega =\Omega$.
\end{proof}

It follows from Lemma \ref{lemma:IntegerCharge} that we can apply Theorem \ref{th:BMTinduction} to the inclusion
$\A_{c_n} \subset \mC_n$ for any positive integer $n$. In particular we can conclude that the $\alpha$-induction
$\alpha^{\mC_n}_{\rho_{\frac{c_n}{3}}}$ of the $\Uone$ automorphism $\rho_{\frac{c_n}{3}}$ is a NS representation of
$\mC_n$.
\begin{lemma}\label{lemma:(n,-n)} $[\alpha^{\mC_n}_{\rho_{\frac{c_n}{3}}}] =(n,-n)$ for every positive integer $n$.
\end{lemma}
\begin{proof}
$\alpha^{\mC_n}_{\rho_{\frac{c_n}{3}}}$ is a NS automorphism of the net $\mC_n$ and hence
$[\alpha^{\mC_n}_{\rho_{\frac{c_n}{3}}}] =(n,m)$ for some $m\in \Z$ such that $|m|\leq n$ and
$n+m \in 2\Z$.
The restriction of $\alpha^{\mC_n}_{\rho_{\frac{c_n}{3}}}$ to $\A_{c_n}$ is a NS representation of the latter
net containing the localized automorphism $\alpha^{\A_{c_n}}_{\rho_{\frac{c_n}{3}}}$ as a subrepresentation.
But the latter is equivalent to $\eta_1$ which by Proposition \ref{prop:eta1} corresponds to the representation
of $\SVir^{N=2}$ with $(h,q) = (\frac{n}{2(n+2)},\frac{n}{n+2})$. It follows that $-\frac{m}{n+2} \in \frac{n}{n+2} +\Z$ and hence that $\frac{2-m}{n+2} \in \Z$. Hence, recalling that $|m|\leq n$, we see that either $m=-n$ or $m=2$.
If $n$ is odd, $m=2$ is forbidden. Now let $n$ be even and greater than $2$.  It follows from Proposition
\ref{prop:charcoset} that the character in the representation $(n,2)$ satisfies
$\chi_{(n,m)}(t) = t^{\frac{n(n+2)-4}{4(n+2)}} +o(t^{\frac{n(n+2)-4}{4(n+2)}})$ and
the equality $[\alpha^{\mC_n}_{\rho_{\frac{c_n}{3}}}] =(n,2)$ would be in contradiction with the fact that
in the representation space of $\alpha^{\mC_n}$ there must be a nonzero vector with conformal energy
$\frac{n}{2(n+2)} < \frac{n(n+2)-4}{4(n+2)}$. Finally, if $n=2$, the equality
$[\alpha^{\mC_n}_{\rho_{\frac{c_n}{3}}}] =(n,2)$ would imply that the representation of $\SVir^{N=2}$
corresponding to the restriction to $\A_{c_n}$ of  $\alpha^{\mC_n}_{\rho_{\frac{c_n}{3}}}$ contains the
irreducible subrepresentations with $(h,q)=(\frac{1}{4},-\frac{1}{2})$ and with
$(h,q)=(\frac{1}{4}, \frac{1}{2})$, in contradiction with $\chi_{(2,2)}(t) = t^{\frac14} +o(t^{\frac14})$.
\end{proof}

\begin{lemma}\label{lemma:etak} For any positive integer $k$ the following hold
\begin{itemize}
\item[-] If $n=4k$ then $[\alpha^{\mC_n}_{\rho_{(2k+1)\frac{c_n}{3}}}]=(n,0)$.
\item[-] If $n=4k-2$ then $[\alpha^{\mC_n}_{\rho_{2k\frac{c_n}{3}}}]=(n,0)$.
\item[-] If $n=4k-1$ then $[\alpha^{\mC_n}_{\rho_{2k\frac{c_n}{3}}}]=(n,-1)$.
\item[-] If $n=4k-3$ then $[\alpha^{\mC_n}_{\rho_{2k\frac{c_n}{3}}}]=(n,1)$.
\end{itemize}
\end{lemma}

\begin{proof}
We use Lemma \ref{lemma:(n,-n)} and the fusion rules in Eq. (\ref{eq:FermiFusions}).
 If $n=4k$ then
\begin{eqnarray*}
[ \alpha^{\mC_n}_{\rho_{(2k+1)\frac{c_n}{3}}} ] &=&  (n,-n)^{2k+1} = (n, -(2k+1)n) = (n, -k(2n+4) + 4k -n) \\
 &=& (n, 0).
\end{eqnarray*}
 If $n=4k-2$ then
\begin{eqnarray*}
[ \alpha^{\mC_n}_{\rho_{2k\frac{c_n}{3}}} ] &=&  (n,-n)^{2k} = (0, -2kn) = (0, -k(2n+4) + n+2)  \\
 &=& (n, 2n + 4) =  (n, 0).
\end{eqnarray*}
 If $n=4k-1$ then
\begin{eqnarray*}
[ \alpha^{\mC_n}_{\rho_{2k\frac{c_n}{3}}} ] &=&  (n,-n)^{2k} = (0, -2kn) = (0, -k(2n+4) + n+1)  \\
 &=& (n, 2n + 3) =  (n, -1).
\end{eqnarray*}
 If $n=4k-3$ then
\begin{eqnarray*}
[ \alpha^{\mC_n}_{\rho_{2k\frac{c_n}{3}}} ] &=&  (n,-n)^{2k} = (0, -2kn) = (0, -k(2n+4) + n+3)  \\
 &=& (n, 2n + 5) =  (n, 1).
\end{eqnarray*}
\end{proof}

We are now ready to prove the main result of this section:
\begin{theorem}\label{th:CosetId} $\A_{c_n} =\mC_n$ for every positive integer $n$.
\end{theorem}
\begin{proof} If $n$ is a positive integer then, by Lemmas \ref{lemma:irred-restr} and \ref{lemma:etak},
there is a $j\in \Z$ such that the restriction to $\A_{c_n}$ of the representation
$\alpha^{\mC_n}_{\rho_{j\frac{c_n}{3}}}$ of $\mC_n$  is irreducible. Now let $\pi^0$ and $\pi_0$ be the vacuum representations of $\mC_n$ and of $\A_{c_n}$, resp., and let $\pi$ be the restriction of $\pi^0$ to $\A_{c_n}$.
The restriction of $\alpha^{\mC_n}_{\rho_{j\frac{c_n}{3}}}$ to $\A_{c_n}$ is
$\pi \circ \alpha^{\A_{c_n}}_{\rho_{j\frac{c_n}{3}}}$ and since $\alpha^{\A_{c_n}}_{\rho_{j\frac{c_n}{3}}}$ is an automorphism we can conclude that $\pi$ is irreducible. Hence $\pi=\pi_0$ and the conclusion follows.
\end{proof}

\section{Classification of $N=2$ Superconformal Nets with $c<3$}\label{sec:classification}

We first look at the bosonic part $\A_c^\gamma$ of the $N=2$ super-Virasoro
net $\A_c$, with $c=3n/(n+2)$ and fixed $n\in\N$, realized as the  coset of ${\mathrm {SU}}(2)_n
\otimes {\mathrm U}(1)_2/{\mathrm U}(1)_{n+2}$ according to Sect. \ref{sec:coset}. We use the notation, labelling, and classification and dimension results for the coset net from Theorem \ref{th:bosonic-coset}.

The fermionic extension arises from the irreducible DHR sector
$(n,n+2,0)$, which has dimension 1, order 2 and 
univalence (statistics phase)
$-1$.

The modular invariants have been classified by Gannon in
\cite[Theorem 4]{Gannon-N2}.  We are interested in the local
extensions of the bosonic part, since its fermionic extension gives
an extension of the original fermionic net, so we need to consider only
the so-called type I modular invariants.  Many of the modular
invariants of Gannon's list are not of type I, so we do not
have to consider them here.  Also, if $(Z_{\lambda\mu})$ is
a type I modular invariant, what we need to check
whether $\bigoplus_{\lambda} Z_{\lambda\mu}\lambda$ is realized
as a dual canonical endomorphism, and if yes, then whether
the realization is unique or not.  (See \cite[Section 4]{KL1}.)

First we deal with the exceptional cases related to the
Dynkin diagrams $E_6$ and $E_8$. Consider the case related to $E_6$,  From Gannon's list of
modular invariants, we see that we consider the following three
cases for $n=10$.

(1) The endomorphism $(0,0,0)\oplus(6,0,0)$.  This is a
dual canonical endomorphism because it arises from
a conformal embedding  ${\mathrm {SU}}(2)_{10}\subset
{\mathrm{SO}}(5)_1$
as in \cite[Section 4]{KL1}, which is a special case of a mirror
extension studied in \cite{X-m}.  For the same reason as in
\cite[Section 4]{KL1} (based on \cite{KL2}), this realization as
a dual canonical endomorphism is unique.  (That is, the $Q$-system
is unique up to unitary equivalence.)

(2) The endomorphism $(0,0,0)\oplus(0,12,0)$.  This is a
dual canonical endomorphism because it arises from
a conformal embedding ${\mathrm U}(1)_{12} \subset {\mathrm U}(1)_3$.
The irreducible DHR sector $(0,12,0)$ has dimension 1 and 
statistics phase
1, so it is realized as a crossed product by
${\mathbb Z}/2{\mathbb Z}$, and hence unique.

(3) The endomorphism
$(0,0,0)\oplus(6,0,0)\oplus(0,12,0)\oplus(6,12,0)$.
This is a combination of the above two extensions.  That is,
we first consider an extension in (1) and make another extension
as a crossed product by ${\mathbb Z}/2{\mathbb Z}$.  For the two
above reasons in (1), (2), we conclude that this realization of
a dual canonical endomorphism is again unique.

The next exceptional case we have to deal with is the case
related to the Dynkin diagram $E_8$, so we now have $n=28$.

We have only one modular invariant here and it gives
the following.

(4) The endomorphism
$(0,0,0)\oplus(10,0,0)\oplus(18,0,0)\oplus(28,0,0)$.  This
arises from a conformal embedding
${\mathrm{SU}}(2)_{28}\subset ({\mathrm G}_2)_1$, and the
realization is unique as in (1).

We now deal with the remaining cases of the modular
invariants.  From  the list of Gannon \cite[Theorem 4]{Gannon-N2},
we see that when we consider the endomorphisms
$\bigoplus_{\lambda} Z_{\lambda\mu}\lambda$ arising from modular
invariants $(Z_{\lambda\mu})$, all the endomorphisms $\lambda$
appearing in this sum have dimensions 1.  If all the irreducible
DHR sectors of a dual canonical endomorphism have dimension 1, then
the extension is a crossed product by a (finite abelian) group
and all these irreducible DHR sectors have 
statistics phase 1.
Hence it is enough to consider only irreducible DHR sectors of
statistics phase 1 and check whether we can construct an extension using them as a
crossed product or not.  We divide the cases depending on $n$.

We start a general consideration.
We first look at the irreducible DHR sectors with dimension $1$.
This condition is equivalent to $l=0,n$.  Then $m$ can be arbitrary
in $\{0,1,\dots, 2n+3\}$, and $s\in\{0,1\}$ is uniquely determined
by the parity of $l+m$, so there are $(4n+8)$ such irreducible DHR sectors.
They give an abelian group of order $4n+8$, and we have a simple
current extension for each subgroup of this group consisting of
irreducible DHR sectors with statistics phase 1, so we need to identify such a subgroup.  This subgroup is clearly a subset of
all the irreducible DHR sectors with statistics phase 1, but this
subset is not a group in general.

If $n\equiv 2$ mod $4$, then the irreducible DHR sectors with dimension $1$
give the group
$({\mathbb Z}/(2n+4){\mathbb Z}) \times ({\mathbb Z}/2{\mathbb Z})$, which
is generated by $\sigma=(0,1,1)$ of order $2n+4$ and
$\tau=(n,0,0)$ of order $2$.

If $n=0$ mod $4$, then the irreducible DHR sectors with dimension $1$
give the group
$({\mathbb Z}/(2n+4){\mathbb Z}) \times ({\mathbb Z}/2{\mathbb Z})$, which is
generated by $\sigma=(0,1,1)$ of order $2n+4$ and
$\tau=(0,n+2,0)$ of order $2$.

If $n$ is odd, the irreducible DHR sectors with dimension $1$
give the group ${\mathbb Z}/(4n+8){\mathbb Z}$, which is
is generated by $\sigma=(0,1,1)$ of order $4n+8$.

The subgroup used for a simple current extension must be a subgroup
of these groups.  If $n$ is odd, it is clearly a cyclic group.
If $n$ is even, it must be a cyclic group or a cyclic group times the
cyclic group of order 2.  If the latter happens, the subgroup must be
of the form $$G\times ({\mathbb Z}/2{\mathbb Z})\subset
({\mathbb Z}/(2n+4){\mathbb Z}) \times ({\mathbb Z}/2{\mathbb Z}),$$
where $G$ is a subgroup of ${\mathbb Z}/(2n+4){\mathbb Z}$.  This
means that the generator for the second component
${\mathbb Z}/2{\mathbb Z}$ must have statistics phase equal to $1$, but
we see that this is not the case for
$n=0,2$ mod $4$.  (In both cases, the statistics phase of the generator for the second component ${\mathbb Z}/2{\mathbb Z}$,
which is $\tau$ in the above notation, is $-1$.)
So also for the case of even $n$, the
subgroup used for a simple current extension must be cyclic.

We next note that if the statistics phase of $(l,m,s)$ is 1 with $s=1$, then $n$ must be a
multiple of $16$ for the following reason.  Suppose
the irreducible DHR sector $(l,m,1)$ has dimension 1 and 
statistics phase 1.
We have $l=0, n$, and first suppose $l=0$.  Then we have
$-2m^2+(n+2)\in 8(n+2){\mathbb Z}$ and $m$ is odd.  Then we first see
$n$ is even, so we set $n=2a$.  Then we have
$-m^2+a+1\in 8(a+1){\mathbb Z}$.  Since $m$ is odd, we know
$-m^2+1=0$ mod $8$.  This implies $a$ is a multiple of $8$, hence
$n$ is a multiple of $16$.  Similarly, we now consider the case
$l=n$.  Then we have
$2n(n+2)-2m^2+(n+2)\in 8(n+2){\mathbb Z}$ and $m+n$ is odd.
This first gives $n$ is even, so we again set $n=2a$.  Then we
have $4a(a+1)-m^2+a+1\in 8(a+1){\mathbb Z}$ and $m$ is odd.
We again have $-m^2+1=0$ mod $8$, so we have that $a$ is a multiple
of $8$.   That is, if $n$ is not a multiple of 16,
we need to consider only the irreducible DHR sectors $(l,m,0)$.

We use the above notations $\sigma, \tau$ for the irreducible
DHR sectors of dimensions 1, and find the maximal cyclic subgroup
which gives a simple current extension.  In general,
its any (cyclic) subgroup also works.

We now consider the following four cases one by one.

[A] Case $n\neq0$ mod $2$.

We need to consider only the irreducible DHR sectors
$(l,m,0)$ with $l+m=0$ mod $2$.  This shows that we need
to consider only the even powers of $\sigma=(0,1,1)$.
Note that $\sigma^2=(n,n+4,0)$ has the statistics phase $\exp(2\pi n i/(2n+4))$.  Then the statistics phase
of $\sigma^{2a}$ is $\exp(2\pi a^2 n i/(2n+4))$.
Consider the set $G$ consisting of $\sigma^{2a}$ with statistics phase 1.  We show that this
set $G$ is a group.  Suppose $\sigma^{2a}$ and $\sigma^{2b}$
are $G$,  Then we have
$a^2n/(2n+4)$ and $b^2n/(2n+4)$ are integers, and we
need to show $(a+b)^2n/(2n+4)$ is also an integer.
It is enough to show that $2abn/(2n+4)$ is an integer.
Let $j_1, j_2, j_3, j_4$ be the numbers of the prime factor 2
in $a,b,n, (n+2)$, resp.~.  We then have
$2j_1+j_3\ge1+j_4$, $2j_2+j_3\ge1+j_4$ and these imply
$j_1+j_2+j_3\ge 1+j_4$.  For an add prime factor $p$, we apply
a similar argument, and we conclude that $2abn/(2n+4)$
is an integer.
Since this set $G$ is a subset of a finite group, this also
shows that $G$ is closed under the inverse operation,
so it is a subgroup of the cyclic group generated by
$\sigma^2$.

We summarize these arguments as follows.
We look for the smallest positive integer $k$ with the statistics phase of $\sigma^k$ equal to $1$.  (Such $k$ is automatically even.)
Then the maximal cyclic subgroup giving a simple current extension
is $\{1,\sigma^k,\sigma^{2k},\dots,\sigma^{4n+8-k}\}$.

[B] Case $n=2$ mod $4$.

We need to consider only $(l,m,0)$ with $l=0,n$ and
$m=0,2,4,\dots,2m+2$.
If $l=0$, then the irreducible DHR sector $(l,m,0)$ has
a statistics phase $1$ if and only if
$m^2/4(n+2)$ is an integer.
If $l=n$, then the irreducible DHR sector $(l,m,0)$ has
a statistics phase $1$ if and only if
$m^2/4(n+2)+1/2$ is an integer.
Consider the set $G$ consisting of
$(0,2m,0)$ with statistics phase $1$.  As in the argument in case [A],
this $G$ is a subgroup of the cyclic group generated by
$(0,2,0)$.  If $(n,m,0)$ has a statistics phase $1$, then
$(0,m,0)$ has a statistics phase $-1$ and thus
$(0,2m,0)$ has a statistics phase $1$, so this is in $G$.
It means now that we need to consider only the odd powers
of the irreducible DHR sector $(n,k/2,0)$, where
$k$ be the smallest even integer with
$(0,k,0)$ having the statistics phase $1$.  The statistics phase of $(0,k/2,0)$ must be one of $-1,\rmi,-\rmi$, since
the statistics phase of $(0,k,0)$ is 1.  If it is $\pm \rmi$,
then all the odd powers of $(0,k/2,0)$ have statistics phase $\pm \rmi$.  If it is $-1$, all the odd powers of $(0,k/2,0)$ have
statistics phase $-1$.   Since the statistics phase of $(n,0,0)$ is
$-1$, in the former case, the maximal cyclic group giving
a simple current extension is
$\{(0,0,0),(0,k,0),(0,2k,0),\dots,(0,2n+4-k,0)\}$, and
in the latter case, the  maximal cyclic group giving
a simple current extension is
$$\{(0,0,0), (n,k/2,0), (0,k,0),(n,3k/2,0),\dots,(n,2n+4-k/2,0)\}.$$

[C] Case $n=4,8,12$ mod $16$.

First note that both irreducible DHR sectors
$(0,0,0)$ and $(n,0,0)$ have statistics phase $1$.

As in [A], the set $G$ of the irreducible DHR
sectors $(0,m,0)$ having a statistics phase $1$ is
a subgroup of the cyclic group ${\mathbb Z}/(2n+4){\mathbb Z}$.
Let $k$ be the smallest positive integer such that
$k^2/(4n+8)$ is an integer.  Then the group $G$ is
given by
$\{(0,0,0),(0,k,0),(0,2k,0),\dots,\\ (0,2n+4-k,0)\}$.
Then the maximal group giving a simple current extension
$$\{(0,0,0),(n,0,0), (0,k,0),(n,k,0),\dots,(0,2n+4-k,0),
(n,2n+4-k,0)\}.$$  Note that this group is isomorphic
to ${\mathbb Z}/((2n+4)/k){\mathbb Z}\times
{\mathbb Z}/2{\mathbb Z}$, but by a general remark above,
this also must be a cyclic group.  This shows that $(2n+4)/k$ is
always odd.

[D] Case $n=0$ mod $16$.

First note that the statistics phase of $\sigma=(0,1,1)$ is
$\exp(\pi ni/(4n+8)$.  The set $G$ consisting of powers of
$\sigma$ with statistics phase $1$ is again a group as in [A].
Let $k$ be the smallest positive integer such that $\sigma^k$
has a statistics phase $1$.  Then $G$ is
$\{1,\sigma^k,\sigma^{2k},\dots,\sigma^{2n+4-k}\}$, where
$1$ stands for the identity sector $(0,0,0)$.
We first show that this $k$ is odd.
Indeed, the condition that the statistics phase is $1$ implies
that $nk^2/(8(n+2))$ is an integer.  Now $n/16$ is an
integer, and $n$ and $n/2+1$ are relatively prime, so
$k^2/(n/2+1)$ is also an integer.  Since $n/2+1$ is odd
and $k$ is the smallest such positive integer, we know
that $k$ is odd.

If $\sigma^t \tau$ has a statistics phase $1$, then $2t$ must
be in the set $\{0,k,2k,\dots,2n+4-k\}$.  Since $k$ is odd,
we have $t\in\{0,k,2k,\dots,2n+4-k\}$.  The statistics phase 
of $\sigma^k \tau$ is
$$\exp\left(2\pi i\left(\frac{-(k+n+2)^2}{4(n+2)}+
\frac{k^2}{8}\right)\right)=1,$$
since $k$ is odd, $n=0$ mod $16$ and
$nk^2/(8n+16)\in{\mathbb Z}$.  All its powers also have
a statistics phase $1$.
We also compute that the statistics phase
of $\sigma^{2ak}\tau$ is equal to
$$\exp\left(2\pi i\left(\frac{-(2ak+n+2)^2}{4(n+2)}+
\frac{4a^2k^2}{8}\right)\right)=-1,$$
since $n=0$ mod $16$ and
$nk^2/(8n+16)\in{\mathbb Z}$.  All these together show that
the set of the irreducible DHR
sectors having statistics phase $1$ is
$$\{1,\sigma^k,\sigma^{2k},\dots,\sigma^{2n+4-k}\}
\cup \{\sigma^k\tau,\sigma^3\tau,\dots,\sigma^{2n+4-k}\tau\}.$$
Note that this set is not a group.
The maximal subgroup giving a simple current extension
is $\{1,\sigma^k,\sigma^{2k},\dots,\sigma^{2n+4-k}\}$ or
$\{1,\sigma^k\tau,\sigma^{2k},\sigma^{3k}\tau,
\dots,\sigma^{2n+4-k}\tau\}$.

\bigskip
We have a few remarks concerning some of the above cases.
First we note that
if $n$ is odd and the statistics phase for $(l,m,s)$ is 1, we have
$l=s=0$.  This shows that all extensions
are mirror extensions arising from extensions
of ${\mathrm U}(1)_{n+2}$.

Second, $n=4,8,12$ mod 16, then all the extensions come from
the cosets of the index 2 extensions of
${\mathrm {SU}}(2)_n$, the mirror
extensions arising from extensions of ${\mathrm U}(1)_{n+2}$ and
the combinations of the two.

Finally, if $n=2$ mod 4, then all the extensions come from the
mirror extensions arising from extensions of ${\mathrm U}(1)_{n+2}$ and
their combinations with $(n,0,0)$ which has statistics phase $-1$.
For example, if $n=6$, the irreducible DHR sector $(6,4,0)$
has dimension 1 and statistics phase 1, and it generates
the cyclic group of order 4.  These four irreducible DHR sectors
are all that have statistics phase 1.  This is the second case
of [B].  Note that the irreducible DHR sectors $(6,0,0)$ and
$(0,4,0)$ both have statistics phase $-1$, so they do not give
a coset construction or a mirror extension, but their combination
$(6,4,0)$ gives a statistics phase $1$.
The case $n=10$ gives the first case of [B].

As an example, consider the case $n=32$.
The smallest $k$ with statistics phase of $\sigma^k$
equal to $1$ is 17.  The statistics phase
of $\sigma^{17}\tau$ is also 1,
so we have two maximal subgroups of order 4,
$\{1,\sigma^{17},\sigma^{34}, \sigma^{51}\}$ and
$\{1,\sigma^{17}\tau,\sigma^{34}, \sigma^{51}\tau\}$.
The union of the two subgroups give a subset of six elements, and
this set gives all the irreducible DHR sectors of dimension 1 and 
statistics phase 1.

\begin{theorem}
The complete list of $N=2$ superconformal nets with $c<3$ in the discrete
series is given as follows:
\begin{itemize}
\item A simple current extension arising from a subgroup of the maximal
cyclic subgroup appearing in the above [A], [B], [C] and [D].
\item The exceptionals related to $E_6$ and $E_8$ as in the above
(1), (2), (3) and (4).
\end{itemize}
\end{theorem}

We finally remark that a cyclic group of an arbitrary order
appears in the above classification.

Suppose an arbitrary positive integer $j$ is given.  We show
that the cyclic group of order $j$ appears in the above [A],
[B], [C] and [D].  (This group is not necessarily maximal.)

We may assume $j\neq1$.
Set $n=j^2-2$.  We are in Case [A] or [B], if $j$ is odd or even,
resp.~.
The irreducible DHR sector $(0,2j,0)$ has a statistics phase $1$,
and the cyclic group
$\{(0,0,0), (0,2j,0), (0,4j,0),\dots (0,2n+4-2j)\}$ gives
a simple current extension.  The order of this group is
$(2n+4)/2j=j$.

\section{Nets of Spectral Triples}\label{sec:STs}

Next we study the supersymmetry properties of the $N=2$  super-Virasoro net $\A_c$ for any of the allowed values of the central charge $c$. Since we would like to construct spectral triples, we may treat this point in a similar way as in \cite[Sect. 4]{CHKL} for the $N=1$ super-Virasoro algebra.

Let $\pi$ be an irreducible general soliton of $\A_c$ with $\rme^{\rmi 4\pi L^\pi_0}$ a scalar. It follows from Theorem \ref{th:SVir-pi-exp-pi} 
(\cf also \cite[Prop. 2.14]{CHL}) that a supercharge, \ie an odd selfadjoint square-root $Q_\pi$ of $L_0^\pi-\const$, exists iff $\pi$ is a Ramond representation though it need not be unique. Then
\[
 Q_{\pi,s} :=\cos(s) G_0^{1,\pi} + \sin(s) G_0^{2,\pi},\quad s\in\R,
\]
are, in fact, possible choices satisfying $Q_{\pi,s}^2 = L_0^\pi-c/24$, as can be checked easily by means of Definition \ref{def:SVir2-alg}.  Moreover, $J_0^\pi$ acts by rotating this supercharge:
 \begin{equation}\label{eq:Q-rot}
  \rme^{\rmi s J_0^\pi} Q_{\pi,0}\rme^{-\rmi s J_0^\pi} = Q_{\pi,s}, \; s\in \R .
 \end{equation}
We let therefore $\pi$ be an irreducible Ramond representation, which is automatically graded according to the discussion after Theorem \ref{th:SVir2-reps}. 
Moreover, we fix $s=0$ and $Q_\pi:=Q_{\pi,0} = G_0^{1,\pi}$. 

Associated to $Q_\pi$, we have in a natural manner a superderivation $(\delta_\pi,\dom(\delta_\pi))$ on $B(\H_\pi)$ as in 
\cite[Sect. 2]{CHKL}, namely
\[
\dom(\delta_\pi) := \{ x\in B(\H_\pi):\; (\exists y\in B(\H_\pi)) \; \gamma(x) Q \subset Q x - y \},
\]
in which case we set $\delta_\pi(x)=y$.
In general, it is difficult to decide whether the local domains $\dom(\delta_\pi)\cap \pi_I(\A_c(I))$ are nontrivial. In \cite[Sect. 4]{CHKL} we discussed this point for the $N=1$ super-Virasoro net, and we shall perform a similar procedure here for the $N=2$ super-Virasoro net $\A_c$.

\begin{theorem}\label{th:SVir2-domain}
Let $\pi$ be an irreducible Ramond representation of $\A_c$. Then for every $I\in\I_\R$, the $*$-subalgebra $\pi_I^{-1}(\dom(\delta_\pi)) \subset \A_c (I)$ is weakly dense. It contains in particular the elements
 \[
 (L(f) +\lambda)^{-1},\;  J(f)(L(f) +\lambda)^{-1},\; G^i(f) (L(f) +\lambda)^{-1},
 \]
if $f\in\Cci_c(I, \R)$ and $I_0\in\I_\R$ are such that $\supp f= \bar{I}_0$, $f(z)>0$ for $z\in I_0$ and $f'(z)\not= 0$ for $z\in I_0$ close to the boundary of $I_0$, and $\lambda\in\C$ with $|\Im \lambda|$ sufficiently large.
\end{theorem}

The proof is quite lengthy and will be subdivided into several lemmata. However, since it borrows many ideas from \cite[Sect. 4]{CHKL}, we can shorten it a bit. We will use the smeared fields associated to the Ramond representation $\pi$ given by Theorem \ref{th:SVir-pi-exp-pi}. We write ``$Q$" instead of ``$Q_\pi$" and ``$\delta$" instead of ``$\delta_\pi$".

\begin{lemma}\label{lm:SVir-core-pres}
For every
$f\in C^\infty(\S , \R)$ and for $\lambda \in \C$ with $|\Im \lambda|$ sufficiently large, we have
\[
(L^\pi(f)+\lambda)^{-1}\Cci(L_0^\pi) \subset \dom((L_0^\pi)^2).
\]
\end{lemma}

This is a special case of \cite[Prop. 4.3]{CHKL}.

\begin{lemma}\label{lem:supportcond}
Let $f\in\Cci(\S, \R)$ such that $\supp f= \bar{I}$, $f(z)>0$ for $z\in I$, and $f'(z)\not= 0$ for $z\in I$ close to the boundary of 
$I$, for some $I \in \I$. Then there is $C>0$ such that $f'^2\le C f$.
\end{lemma}
The proof is given in \cite[Lemma 4.7]{CHKL}.

\begin{lemma}
\label{lm:LGJ-basic}
Let $f\in\Cci(\S,\R)$. Then, for $i=1,2$ and $\lambda\in\C$ with $|\Im\lambda|$ sufficiently large, we have
\[
G^{i,\pi}(f)(L^\pi(f)+\lambda)^{-1} \in B(\H_\pi).
\]
Furthermore, if $f$ satisfies the condition in the preceding Lemma \ref{lem:supportcond}, then
\[
G^{i,\pi}(f')(L^\pi(f)+\lambda)^{-1} \in B(\H_\pi).
\]
\end{lemma}
This is an immediate consequence of \cite[Prop. 4.6 \& Lemma 4.7]{CHKL}.

 \begin{lemma}\label{lem:L-domain}
Let $f\in\Cci(\S,\R)$ satisfy the conditions in Lemma \ref{lem:supportcond} and assume $\lambda\in\C$ with $|\Im\lambda|$ sufficiently large. Then
 $(L^\pi(f) + \lambda)^{-1}\in\dom(\delta)$, and
 \begin{align*}
 \delta\big((L^\pi(f)+\lambda)^{-1} \big)= &
 -\frac{\rmi}2 (L^\pi(f)+\lambda)^{-1}G^{1,\pi}(f')(L^\pi(f)+\lambda)^{-1},
  \end{align*}
  \end{lemma}
The proof goes as in \cite[Thm. 4.8]{CHKL}.

\begin{lemma}\label{lem:G-domain}
Let $f\in\Cci(\S, \R)$ satisfy the condition in Lemma \ref{lem:supportcond}, and $\lambda\in\C$ with imaginary part sufficiently large. Then
 \[
  G^{i,\pi}(f)(L^\pi(f) + \lambda)^{-1} ,\; i=1,2,
 \]
 are in $\dom(\delta)$, and
 \begin{align*}
 \delta\big(G^{1,\pi}(f)(L^\pi(f)+\lambda)^{-1} \big) =&
 \left( 2 L^\pi(f) -\frac{c}{24\pi}\int_{S^1}f  \right)(L^\pi(f)+\lambda)^{-1} \\
 &+ \frac{\rmi}2G^{1,\pi}(f)(L^\pi(f)+\lambda)^{-1}G^{1,\pi}(f')(L^\pi(f)+\lambda)^{-1},\\
 \delta\big(G^{2,\pi}(f)(L^\pi(f)+\lambda)^{-1} \big) =&
 -J^\pi(f') (L^\pi(f)+\lambda)^{-1} \\
 &+ \frac{\rmi}2 G^{2,\pi}(f)(L^\pi(f)+\lambda)^{-1}G^{1,\pi}(f')(L^\pi(f)+\lambda)^{-1}.
 \end{align*}
  \end{lemma}

For the proof we shall need \emph{local energy bounds} \cite{CW}: instead of bounding $J^\pi(f)$, for $f\in\Cci(\S, \R)$, by a multiple of $\unit+L_0^\pi$ as in \eqref{eq:LEB1}, it may be bounded by smeared fields, namely there are scalars $C_1,C_2>0$ depending on $f$ such that
\begin{equation}\label{eq:LEB2}
J^\pi(f)^2 \le C_1\unit + C_2 L^\pi(f^2).
\end{equation}

\begin{proof}
The fact that $(L^\pi(f) + \lambda)^{-1}$ preserves $\dom((L_0^\pi)^2)$ has been obtained in Lemma \ref{lm:SVir-core-pres}, so $ G^{i,\pi}(f)(L^\pi(f) + \lambda)^{-1}$ and $J^\pi(f')(L^\pi(f) + \lambda)^{-1}$ map $\Cci(L_0^\pi)$ into $\dom(L_0^\pi)$ by a standard application of the linear energy bounds \eqref{eq:LEB1}.

The case $i=1$ corresponds to \cite[Thm. 4.11]{CHKL} because of our choice $Q=G_0^{1,\pi}$ here.
Concerning the case $i=2$, the commutation relations in Definition \ref{def:SVir2-alg} imply, for $\psi\in\Cci(L_0^\pi)$:
\begin{align*}
 Q G^{2,\pi}(f)(L^\pi(f)+\lambda)^{-1}\psi 
 =& - G^{2,\pi}(f) Q (L^\pi(f)+\lambda)^{-1}\psi
- J^\pi(f') (L^\pi(f)+\lambda)^{-1}\psi\\
=& - G^{2,\pi}(f)(L^\pi(f)+\lambda)^{-1}Q\psi \\
 &+ \frac{\rmi}2
 G^{2,\pi}(f)(L^\pi(f)+\lambda)^{-1}G^{1,\pi}(f')(L^\pi(f)+\lambda)^{-1}\psi
 \\
 &- J^\pi(f') (L^\pi(f)+\lambda)^{-1}\psi.
\end{align*}

We know from the assumptions of this lemma and Lemma \ref{lm:LGJ-basic} that both $G^{i,\pi}(f)(L^\pi(f)+\lambda)^{-1}$ and $G^{i,\pi}(f')(L^\pi(f)+\lambda)^{-1}$ are bounded. This shows that the operator in the second term on the right-hand side extends to a bounded operator on $\H_\pi$. 

By making use of local energy bounds \eqref{eq:LEB2} for currents, we obtain $C_1,C_2>0$ such that
\[
J^\pi(f')^2\le C_1\unit + C_2 L^\pi(f'^2) \le C_1'\unit + C_2' L^\pi(f).
\]
The second inequality is obtained from the fact that $f'^2\le Cf$ for some $C>0$ by assumption, whence $L^\pi(Cf-f'^2)$ is bounded from below \cite[Thm. 4.1]{FH}, say by $-C'\unit$. Then we may choose $C_2':=C_2 C$ and $C_1':=C_1+C_2 C'$. Consequently, for every $\psi\in\Cci(L_0^\pi)$, we have
\begin{align*}
\|J^\pi(f')&(L^\pi(f)+\lambda)^{-1}\psi\|^2 = \langle \psi, (L^\pi(f)+\bar{\lambda})^{-1} J^\pi(f')^2(L^\pi(f)+\lambda)^{-1}\psi \rangle\\
\le& C_1' \|(L^\pi(f)+\lambda)^{-1}\psi\|^2 + C_2'  \|L^\pi(f) (L^\pi(f)+\lambda)^{-1}\psi\| \cdot \|(L^\pi(f)+\lambda)^{-1}\psi\|,
\end{align*}
so $J^\pi(f')(L^\pi(f)+\lambda)^{-1}$ is bounded.
\end{proof}

\begin{lemma}\label{lem:J-domain}
Let $f\in\Cci(\S,\R)$ satisfy the condition in Lemma \ref{lem:supportcond}, and let $\lambda\in\C$ with imaginary part sufficiently large. Then
 $ J^\pi(f)(L^\pi(f) + \lambda)^{-1} \in \dom(\delta)$, and
\begin{align*}
\delta\big(J^\pi(f)(L^\pi(f)+\lambda)^{-1} \big)= &
- \rmi G^{2,\pi}(f)(L^\pi(f)+\lambda)^{-1} \\
&+ J^\pi(f)(L^\pi(f)+\lambda)^{-1}\frac{\rmi}{2}G^{1,\pi}(f')(L^\pi(f)+\lambda)^{-1}.
\end{align*}
\end{lemma}

\begin{proof}
The fact that $ J^\pi(f)(L^\pi(f) + \lambda)^{-1}$ maps $\Cci(L_0^\pi)$ into $\dom(L_0^\pi)$ for $|\Im \lambda|$ sufficiently large has been obtained in Lemma \ref{lm:SVir-core-pres}. Let us show the boundedness.

Notice first that there is $C$ such that $f^2\le Cf$. Using then the local energy bounds in \eqref{eq:LEB2} as in the proof of Lemma \ref{lem:G-domain}, we find constants $C_1',C_2'>0$ such that
\[
J^\pi(f^2) \le C_1' \unit + C_2' L^\pi(f),
\]
and analogously, we obtain, for $\psi\in\Cci(L_0^\pi)$:
\begin{align*}
\|J^\pi(f)& (L^\pi(f)+\lambda)^{-1}\psi\|^2\\  
\le& C_1' \|(L^\pi(f)+\lambda)^{-1}\psi\|^2 + C_2'  \|L^\pi(f) (L^\pi(f)+\lambda)^{-1}\psi\| \cdot \|(L^\pi(f)+\bar{\lambda})^{-1}\psi\|,
\end{align*}
hence the boundedness of $J^\pi(f)(L^\pi(f)+\lambda)^{-1}$.

We already know from Lemma \ref{lm:LGJ-basic} that under the present assumptions on $f$, the operators $G^{i,\pi}(f')(L^\pi(f)+\lambda)^{-1}$ are bounded and map $\Cci(L_0^\pi)$ into $\dom(L_0^\pi)$. Thus, using Definition \ref{def:SVir2-alg}, we compute:
\begin{align*}
Q J^\pi(f)(L^\pi(f)+\lambda)^{-1}\psi =& J^\pi(f) Q (L^\pi(f)+\lambda)^{-1} \psi
- \rmi G^{2,\pi}(f') (L^\pi(f)+\lambda)^{-1}\psi\\
=& J^\pi(f) (L^\pi(f)+\lambda)^{-1} Q \psi
- \rmi G^{2,\pi}(f') (L^\pi(f)+\lambda)^{-1}\psi \\
&+ J^\pi(f)(L^\pi(f)+\lambda)^{-1}\frac{\rmi}{2}G^{1,\pi}(f')(L^\pi(f)+\lambda)^{-1} \psi.
\end{align*}
This completes the proof using Lemma \ref{lm:LGJ-basic}.
\end{proof}

\begin{proofof}[Theorem \ref{th:SVir2-domain}] 
Thanks to Theorem \ref{th:SVir-pi-exp-pi} and the local normality of $\pi$, the proof goes now almost precisely as in \cite[Lemma 4.12]{CHKL}, but for the sake of completeness and because of its importance we present it here again. Let $I\in \I_\R$. Then 
$\pi_I^{-1}(\dom(\delta))$ is a unital $*$-subalgebra of $\A_c(I)$ wherefore, by the von Neumann density
 theorem, it suffices to show that
\[
 \pi_I^{-1}(\dom(\delta))' \subset \A_c(I)'.
 \]
To this end let
 $f$ be an arbitrary real smooth function  with support in $I$.
 Recalling that $I$ must be open it is easy to see that there is an
 interval $I_0 \in \I$ such that  $\overline{I_0} \subset I$ and
 $\supp f \subset I_0$  and a smooth function
 $g$ on $\S$ such that $\supp g\subset \overline{I_0}$,
 $g(z)>0$ for all $z \in I_0$, $g'(z)\neq 0$ for all $z\in I_0$
 sufficiently close to the boundary and $g(z)=1$ for all $z \in \supp f$.  Accordingly, there is a real number $s>0$ such that
 $f(z)+sg(z) > 0$ for all $z \in I_0$. Now let $f_1=f+sg$ and
 $f_2=sg$. Then $f=f_1-f_2$. Hence it follows from the above lemmata and the definition of $\A_c(I)$ that,
 for $|\Im\lambda|$ sufficiently large, all the operators
\[
(L^\pi(f_j)+\lambda)^{-1},\; J^\pi(f_j)(L^\pi(f_j)+\lambda)^{-1},\;
  G^{i,\pi}(f_j)(L^\pi(f_j)+\lambda)^{-1},\quad j=1,2,
\]
belong to $\dom(\delta)$. Thus according to Theorem \ref{th:SVir-pi-exp-pi}, all the operators
\[
(L(f_j)+\lambda)^{-1},\; J(f_j)(L(f_j)+\lambda)^{-1},\;
  G^i(f_j)(L(f_j)+\lambda)^{-1},\quad j=1,2,
\]
belong to $\pi_I^{-1}(\dom(\delta))$. 
 So if $a \in \pi_I^{-1}(\dom(\delta))' $, then $a$ commutes with $L(f_j)$, $J(f_j)$, and $G^i(f_j)$, $j=1,2$. Therefore,
 if $\psi_1, \psi_2 \in C^\infty(L_0)$ then,
 \begin{align*}
 (a\psi_1,L(f)\psi_2) 
 &= (a\psi_1,L(f_1)\psi_2)-(a\psi_1,L(f_2)\psi_2) \\
 &= (aL(f_1)\psi_1,\psi_2)-(aL(f_2)\psi_1,\psi_2) \\
 &= (aL(f)\psi_1,\psi_2)
 \end{align*}
 and, since $C^\infty(L_0)$ is a core for $L(f)$, it follows that
 $a$ commutes with $L(f)$ and hence with $\rme^{\rmi L(f)}$. Similarly
 $a$ commutes with $\rme^{\rmi G^i(f)}$ and $\rme^{\rmi J(f)}$. Hence $a \in \A_c(I)'$ and the
 statement follows. 
 \end{proofof}

To conclude this section then, recall from \cite[Def. 3.9]{CHKL}:

\begin{definition}
A \emph{net of graded spectral triples} $(\AA(I) ,(\pi_I,\H), Q_\pi)_{I\in \I_\R}$ over 
$\S\setminus\{-1\}\simeq \R$ 
consists of a graded Hilbert space $\H$, a selfadjoint operator $Q_\pi$, and
a net $\AA$ of unital $*$-algebras on $\I_\R$ acting on $\H$ via the graded general soliton $\pi$,
\ie a map from $\I_\R$ into the family of unital $*$-algebras represented on
$B(\H)$ which satisfies the isotony property
$$\AA(I_1)\subset \AA(I_2)\quad {\rm if}\; I_1\subset I_2,$$
and such that $(\AA(I) , (\pi_I,\H), Q_\pi)$ is a graded spectral triple for all $I\in \I_\R$.
\end{definition}

\begin{corollary}\label{cor:SVir2-STs}
Let $\pi$ be an irreducible Ramond representation of $\A_c$. Then, setting $\AA(I):= \pi_I^{-1}(\dom(\delta^\pi))$, the family 
\[
 (\AA(I),(\pi,\H_\pi),Q_\pi)_{I\in\I_\R}
\] 
forms a nontrivial net of even $\theta$-summable graded spectral triples over $\S\setminus\{-1\}$.
\end{corollary}

\section{JLO Cocycles, Index Pairings, and Ramond Sectors for the $N=2$ Super-Virasoro nets}\label{sec:JLO}

In the preceding section, in particular in Corollary \ref{cor:SVir2-STs}, we constructed in a canonical manner a nontrivial net of spectral triples for a given Ramond representation of any $N=2$ super-Virasoro net. These spectral triples give rise to local JLO cocycles and thus define an index pairing as constructed in \cite{CHL}. In this section, in order to capture the global aspects of superselection theory, we will define for any $N=2$ super-Virasoro net $\A_c$ an appropriate global locally convex algebra $\AA_c$ such that every irreducible Ramond representation of the net $\A_c$ gives rise to a spectral triple (K-cycle) for $\AA_c$. The nets of spectral triples in the previous section arise by restriction to the local subalgebras of $\AA_c$. Moreover we will show that for inequivalent Ramond representations the corresponding JLO cocycles belong to different entire cyclic cohomolgy classes. 
Although there are some important differences, the strategy adopted here is similar to the one recently used in \cite{CHL}. In particular, as in \cite[Sect. 5]{CHL}, we will use the ``characteristic projections'' in order to distinguish different cohomology classes through the pairing with K-theory.  For unexplained notions of noncommutative geometry we refer to Connes' book \cite{Co1}. All notions of noncommutative geometry needed in this section can also be found in the brief overview given in \cite[Sect. 3]{CHL}.

Let $\A_c$ be the $N=2$ super-Virasoro net with central charge $c$ and let $W^*(\A_c^\gamma)$ be the universal von Neumann algebra from Definition \ref{def:CFT-univC}. Then every representation $\pi$ of $\A_c^\gamma$ can be identified with a normal representation of $W^*(\A_c^\gamma)$ which will also be denoted by $\pi$ and vice versa. In particular, every general $\PSL$-covariant soliton $\pi$ of the graded-local conformal net $\A_c$ gives rise by restriction to a representation of $\A_c^\gamma$ and hence to a normal representation of $W^*(\A_c^\gamma)$ which,  when no confusion can arise, we will denote again by $\pi$. 

\begin{definition}\label{def:JLO-AA}
Let $\Delta^c_R$ be a maximal family of mutually inequivalent irreducible Ramond representations of $\A_c$. 
The \emph{differentiable global algebra} associated with the local conformal net $\A_c^\gamma$ is the unital $*$-algebra defined as
\[
 \AA_c := \{a\in W^*(\A_c^\gamma):\; (\forall \pi\in\Delta^c_R)\; \pi(a)\in \dom(\delta_\pi)\}.
\]
The corresponding local subalgebras are defined by $\AA_c (I) := \AA_{c} \cap \A^\gamma(I)$. Endowed with the family of norms
$$\| \cdot \|_\pi := \| \cdot \|_{W^*(\A_c^\gamma)} + \| \delta_\pi(\pi(\cdot))\|_{B(\H_\pi)}, \quad \pi \in \Delta^c_R,$$
 $\AA_c$ becomes a locally convex algebra. Here, as in Sect. \ref{sec:STs} $\delta_\pi$ denotes the superderivation induced by 
 the supercharge operator $Q_\pi := G_0^{1,\pi}$.  
\end{definition}

\begin{remark}\label{remarkAAc} It is straightforward to see that neither the algebra nor the corresponding family of norms (and hence the corresponding locally convex topology) depend on the choice of $\Delta^c_R$. In fact each norm $\| \cdot \|_\pi$ only depends on the unitary equivalence class $[\pi]$ of $\pi$. As a consequence we have 
$\pi(\AA_c) \subset \dom(\delta_\pi)$ for every irreducible Ramond representation of $\A_c$. Moreover, $\AA_c$ is nontrivial and in fact by 
Theorem \ref {th:SVir2-domain} (and its proof) $\AA_c(I)$ is weakly dense in $\A_c^\gamma(I)$ for all $I\in \I$.
If $c<3$ then, $\Delta^c_R$ is finite and the locally convex topology on $\A_c$ can be induced by the norm 
$\|\cdot \|_R := \sum_{\pi \in \Delta^c_R} \|\cdot\|_\pi .$

\end{remark}

Now, recall that every irreducible Ramond representation $\pi$ of $\A_c$ is graded by $\Gamma_\pi =\rme^{-i \pi q_\pi}  \rme^{i\pi J_0}$, 
that $Q_\pi^2 = L_0^\pi - {c}/{24}$,that $\rme^{-\beta L_0^\pi}$ is a trace class operator for every $\beta >0$ 
and that we can consider $\pi$ as a representation of $W^*(\A_c^\gamma)$ and hence, by restriction, 
as a representation of $\AA_c$. Moreover, by definition of the locally convex algebra $\AA_c$, the map $\pi: \AA_c \to \dom{\delta_\pi}$ is continuous when the latter is endowed with the Banach algebra norm $\| \cdot \|_1 := \| \cdot\|_{B(\H_\pi)} + \| \delta_\pi(\cdot)\|_{B(\H_\pi)}$.  Moreover if $\pi_1$ is unitarily equivalent to $\pi$ and the unitary equivalence is realized through a unitary intertwiner 
$u: \H_\pi \to \H_{\pi_1}$ then $u\Gamma_\pi u^* = \Gamma_{\pi_1}$ and $u Q_\pi u^* = Q_{\pi_1}$.
As a consequence we have the following (cf. \cite[Theorem 4.10]{CHL}):  

\begin{proposition}\label{prop:JLO-ST}
For every irreducible Ramond representation  $\pi$ of the net $\A_c$, the data $(\AA_c,(\pi,\H_\pi,\Gamma_\pi),Q_\pi)$ is a nontrivial $\theta$-summable even spectral triple such that $\pi: \AA_c \to \dom(\delta_\pi)$ is continuous. Accordingly the associated JLO cocycle $\tau_\pi$ is a well-defined even entire cyclic cocycle of the locally convex algebra $\AA_c$. If $\pi_1 \simeq \pi$ then $\tau_{\pi_1} = \tau_\pi$.
\end{proposition}

For every irreducible Ramond representation $\pi$ of $\A_c$ we have the direct sum decomposition 
$\H_\pi= \H_{\pi, +} \oplus \H_{\pi_ ,-}$, where $ \H_{\pi, \pm}$ is the eigenspace of $\Gamma_\pi$ corresponding to the eigenvalue 
$\pm 1$ and the corresponding decomposition of $\pi |_{\A_c^\gamma}$ into a direct sum of irreducible representations $\pi_+\oplus \pi_-$,
cf. Proposition  \ref {prop:CKL22}. Since $Q_\pi$ is an odd operator, its domain $\dom (Q_\pi)$ is preserved by the action of $\Gamma_\pi$ and hence it decomposes as a direct sum  $\dom (Q_\pi) = \dom (Q_\pi)_+ \oplus \dom (Q_\pi)_-$, with 
$\dom (Q_\pi)_\pm$ dense in  $\H_{\pi, \pm}$ and therefore there are operators $Q_{\pi, \pm}$ from $\H_{\pi, \pm}$ into 
$\H_{\pi, \mp}$ with dense domains $\dom(Q_\pi)_\pm \subset\H_{\pi, \pm}$ satisfying 
$Q_\pi \left( \psi_+ \oplus \psi_- \right) = Q_{\pi, -}\psi_- \oplus Q_{\pi, +}\psi_+$ for $\psi_\pm \in  \dom(Q_\pi)_\pm$, and 
 $Q_{\pi , \pm}^* = Q_{\pi , \mp}$.  
Now, recall from \cite{Co1} (see also \cite[Sect. 3]{CHL}) that for every projection $p\in \AA_c$ the operator 
$\pi_-(p)Q_{\pi, +} \pi_+(p) :  \pi_+(p)\H_{\pi, +} \to \pi_-(p) \H_{\pi, -}$ is a Fredholm operator and that the integer
$\tau_\pi(p):= \operatorname{ind}\left(\pi_-(p)Q_{\pi, +} \pi_+(p) \right)$ ({\it index pairing}) only depends on the entire cyclic cohomology class of $\tau_\pi$ and on the class of $p$ in the 
K-theory group $K_0(\AA_c)$.  We want to use this fact to show that the JLO cocycles $\tau_\pi$ give rise to distinct entire cyclic cohomology
classes for inequivalent irreducible Ramond representations $\pi$. To this end we need to consider appropriate projections in $\AA_c$ for which the index computations are easy enough and give the desired result. A similar strategy has been used in \cite[Sect. 5]{CHL}. 

We first note that if $\pi_-(p)=0$ then we have $\tau_\pi (p)= \mathrm{dim}(\pi_+(p)\H_{\pi, +})$. Now, given a representation $\pi$ of 
$C^*(\A_c^\gamma)$ which is quasi-equivalent to a subrepresentation of the universal representation $\tilde{\pi}_u$, denote by 
$s(\pi)\in \mathcal{Z}(W^*(\A_c^\gamma))$ the central support of the projection onto this subrepresentation so that, in 
particular $\pi(s(\pi)) = \unit$. If $\pi$ is an irreducible Ramond representation of $\A_c$ we will write $s(\pi)$ instead of 
$s(\pi |_{\A_c^\gamma})$. With this convention we have $s(\pi)= s(\pi_+) + s(\pi_-)$.  
 Note that the unique normal extension of $\pi$ to $W^*(\A_c^\gamma)$ gives rise to an isomorphism 
of  $W^*(\A_c^\gamma)s(\pi)$ onto $\pi(C^*(\A_c^\gamma))''$. 

\begin{definition}\label{def:CharProj}
Given a graded irreducible Ramond representation $\pi$ of $\A_c$, let $s(\pi_+)\in \mathcal{Z}(W^*(\A_c^\gamma))$ be the central projection corresponding to the irreducible subrepresentation $\pi_+$ of $\pi |_{\A_c^\gamma}$ on the subrepresentation $\pi$. Moreover, write $p_{h_\pi, +} \in B(\H_\pi)$ for the projection onto the (finite-dimensional) lowest energy subspace of $\H_{\pi, +}$. Then there is a unique projection $p_\pi \in W^*(\A_c^\gamma)s(\pi)$ such that $\pi(p_\pi)=p_{h_\pi, +}$ and we call it the \emph{characteristic projection} of $\pi$.
\end{definition}

\begin{lemma}\label{lem:CharProj}
 Let $\pi_1$ and $\pi_2$ be two irreducible Ramond representations of $\A_c$. Then 
$s(\pi_1)=s(\pi_2)$ and $p_{\pi_1} = p_{\pi_2}$ if $\pi_1$ and $\pi_2$ are unitarily equivalent, while $s(\pi_1)s(\pi_2)=0$ and 
$p_{\pi_1} p_{\pi_2} = 0$ otherwise. 
\end{lemma}
\begin{proof} If $\pi_1$ and $\pi_2$ are unitarily equivalent we clearly have $s(\pi_1)=s(\pi_2)$. If the unitary equivalence is realized by a unitary intertwiner $u: \H_{\pi_1} \to \H_{\pi_2}$ then $u\Gamma_{\pi_1}u^* = \Gamma_{\pi_2}$ and hence 
$u: \H_{\pi_1, +} \to \H_{\pi_2, +}$ so that $\pi_{1, +}$ and $\pi_{2, +}$ are unitarily equivalent and therefore the central supports $s(\pi_{1, +})$ and $s(\pi_{2, +})$ are equal. Moreover, $uL_0^{\pi_1}u^*=L_0^{\pi_2}$ and hence $up_{h_{\pi_1}, +}u^* = p_{h_{\pi_2}, +}$. Accordingly $\pi_2(p_{\pi_1}) = u\pi_1(p_{\pi_1})u^* = up_{h_{\pi_1}, +}u^* = p_{h_{\pi_2}, +}$ and hence $p_{\pi_1} = p_{\pi_2}$. 
On the other hand, if $\pi_1$ and $\pi_2$ are inequivalent, then, by Proposition  \ref{prop:CKL22}, $\pi_{1,+} \oplus \pi_{1,-}$ is disjoint from 
$\pi_{2,+} \oplus \pi_{2,-}$. Accordingly $s(\pi_1)s(\pi_2)=0$ and $p_{\pi_1} p_{\pi_2} = p_{\pi_1} s(\pi_1)s(\pi_2) p_{\pi_2} = 0$.
\end{proof}
\begin{proposition}\label{prop:CharProj} Let $\pi$ be an irreducible Ramond representation of $\A_c$. Then, for every irreducible Ramond 
representation $\tilde{\pi}$ of $\A_c$, we have 
\[
\tilde{\pi}(p_\pi) =
\left\lbrace \begin{array}{l@{\quad \mathrm{if} \quad }l}
p_{h_{\tilde{\pi}}, +} & \tilde{\pi} \simeq\pi \\
0 & \tilde{\pi} \not\simeq \pi
\end{array} \right. 
\]
and 
\[
\tilde{\pi}(s(\pi)) =
\left\lbrace \begin{array}{l@{\quad \mathrm{if} \quad }l}
\unit & \tilde{\pi} \simeq\pi \\
0 & \tilde{\pi} \not\simeq \pi.
\end{array} \right. 
\]
Moreover, $p_\pi$ and $s(\pi)$ belong to $\AA_c$. 
\end{proposition} 
\begin{proof}
If $\tilde{\pi} \simeq \pi$ then, by Lemma \ref{lem:CharProj} we have $s(\tilde{\pi})=s(\pi)$ and $p_{\tilde{\pi}} = p_{\pi}$ and hence
$\tilde{\pi}(s(\pi))=\tilde{\pi}(s(\tilde{\pi})) =\unit$ and $\tilde{\pi}(p_\pi) =    \tilde{\pi}(p_{\tilde{\pi}}) =p_{h_{\tilde{\pi}}, +}$. 
If $\tilde{\pi} \not\simeq \pi$ then, again by Lemma \ref{lem:CharProj},  $s(\tilde{\pi})s(\pi)=0$ and hence $\tilde{\pi}(s(\pi))=0$ and 
$\tilde{\pi}(p_\pi) = \tilde{\pi}(s(\pi)p_\pi) =0$. Now, $\unit$ and $0$ clearly belong to the domain of the superderivation $\delta_{\tilde{\pi}}$. 
Accordingly $\tilde{\pi}(s(\pi))  \in \dom(\delta_{\tilde{\pi}})$, for every irreducible Ramond representation $\tilde{\pi}$ of $\A_c$, and hence
$s(\pi) \in \AA_c$. 
Moreover, since $Q_{\tilde{\pi}}^2=L_0^{\tilde{\pi}} - \frac{c}{24}\unit $ and since $p_{h_{\tilde{\pi}}, +}$ commutes with $L_0^{\tilde{\pi}}$ and has finite-dimensional range, we see that $Q_{\tilde{\pi}}p_{h_{\tilde{\pi}}, +}$ is everywhere defined and bounded and hence 
$p_{h_{\tilde{\pi}}, +}Q_{\tilde{\pi}}$ is bounded on the domain of $Q_{\tilde{\pi}}$. It follows that 
$p_{h_{\tilde{\pi}}, +} \in \dom(\delta_{\tilde{\pi}})$. Therefore, for every irreducible Ramond representation $\tilde{\pi}$ of $\AA_c$, we have $\tilde{\pi}(p_\pi) \in \dom(\delta_{\tilde{\pi}})$, so $p_\pi \in \AA_c$ and we are done. 
\end{proof}

With all these ingredients at hand, one can now easily prove the main result of this section, cf. also \cite{CHL}.

\begin{theorem}\label{th:SVir2-index} 
Let $\pi_1$ , $\pi_2$ be irreducible Ramond representations of $\A_c$. Then the index pairing between JLO cocycle $\tau_{\pi_1}$
and the projection $p_{\pi_2} \in \AA_c$ gives 

\[
\tau_{\pi_1} (p_{\pi_2}) = \left\lbrace \begin{array}{l@{\quad \mathrm{if} \quad }l}
                                             1 & \pi_1 \simeq \pi_2 \\
					     0 & \pi_1 \not\simeq \pi_2 .
                                             \end{array}
\right.
\]
Hence, the map $\pi \to [\tau_\pi]$, where $[\tau_\pi]$ denotes the entire cyclic cohomology class of the JLO cocycle $\tau_\pi$ associated to the irreducible Ramond representation $\pi$ of $\A_c$, gives a complete noncommutative geometric invariant for the class of irreducible Ramond representations of the $N=2$ super-Virasoro net $\A_c$, namely $[\tau_{\pi_1}] = [\tau_{\pi_2}] $ if and only if 
$\pi_1 \simeq \pi_2$.
\end{theorem}

\section{Ramond Vacuum Representations and Chiral Rings}\label{sec:chiral-fusion}

Let $\pi$ be an irreducible Ramond representation of the $N=2$ super-Virasoro net $\A_c$. Then we say that $\pi$ is a 
{\it Ramond vacuum representation}
if the corresponding irreducible representation of the $N=2$ Ramond super-Virasoro algebra is a Ramond vacuum representation as defined in Section \ref{sec:SVir-net}, namely if the corresponding lowest energy $h_\pi$ satisfies $h_\pi=c/24$. We denote by 
$\Delta^c_{R\mathrm{vac}}$ the set of equivalence classes of Ramond vacuum representations of $\A_c$. For $c<3$ the set 
$\Delta^c_{R\mathrm{vac}}$ is finite. The following lemma shows that the Ramond vacuum representations can be characterized in terms of the associated JLO cocycles and the corresponding index pairing. 

\begin{lemma}\label{lemmaRvac} 
Let $\pi$ be an irreducible Ramond representation of the $N=2$ super-Virasoro net $\A_c$. Then we have $\tau_\pi(s(\pi))=1$ 
if $\pi$ is a Ramond vacuum representation, and  $\tau_\pi(s(\pi))=0$ otherwise. 
\end{lemma}
\begin{proof} Recall that for every irreducible Ramond representation of $\A_c$ we have $\pi(s(\pi)) =\unit$ and hence 
$\tau_\pi(s(\pi))= \tau_\pi(1)$. It follows that  
\begin{eqnarray*}
\tau_\pi(s(\pi)) &=& \operatorname{ind}\left(Q_{\pi, +}\right)  = \mathrm{dim}(\ker (Q_{\pi, +})) -  \mathrm{dim}(\ker (Q_{\pi, -})) \\
&=& \mathrm{dim}(\ker (Q_{\pi, -}Q_{\pi, +})) -  \mathrm{dim}(\ker (Q_{\pi, +}Q_{\pi, -})) \\ 
&=& \mathrm{dim}\left( \ker \left( (L_0^\pi - \frac{c}{24} \unit)|_{\H_{\pi, +}} \right) \right) -  
\mathrm{dim}\left( \ker \left( (L_0^\pi - \frac{c}{24}\unit)|_{\H_{\pi, -}} \right) \right).
\end{eqnarray*}
Now, if $\pi$ is not a Ramond vacuum representation, we have 
$\mathrm{dim}\left( \ker \left(L_0^\pi - \frac{c}{24} \unit \right) \right) =0$ and hence $\tau_\pi(s(\pi))=0$. On the other hand, 
if $\pi$ is a Ramond vacuum representation, we see from the discussion before Theorem \ref{th:SVir2-reps}
that $\mathrm{dim}\left( \ker \left( (L_0^\pi - \frac{c}{24} \unit)|_{\H_{\pi, +}} \right) \right)=1$ and 
$\mathrm{dim}\left( \ker \left( (L_0^\pi - \frac{c}{24}\unit)|_{\H_{\pi, -}} \right) \right)=0$, so $\tau_\pi(s(\pi))=1$.
\end{proof}

It follows from the previous lemma that, in order to separate the entire cyclic cohomology classes of the JLO cocycles associated to Ramond vacuum representations, we can consider the central support projections $s(\pi)$ instead of the characteristic projections $p_\pi$ used in the previous section. Note that these generate a $*$-subalgebra $\AA_{c, {\Delta^c_{R\mathrm{vac}}}}$ of the center $\mathcal{Z}(\AA_c) =\AA_c \cap \mathcal{Z}(W^*(\AA_c))$. If $c <3$ then $\Delta^c_{R\mathrm{vac}}$ is finite and the above commutative subalgebra is finite-dimensional and isomorphic to $\C^{\Delta^c_{R\mathrm{vac}}}$. We record these facts in the following:

\begin{proposition}\label{prop:JLO-centre}
If $c<3$ the central support projections associated to the Ramond vacuum representations of $\A_c$ generate a finite-dimensional $*$-subalgebra  $\AA_{c, {\Delta^c_{R\mathrm{vac}}}}$ of $ \mathcal{Z}(\AA_c)$ isomorphic to the abelian $*$-algebra 
$\C^{\Delta^c_{R\mathrm{vac}}}$. The restriction of the JLO cocycle $\tau_\pi$ associated to the Ramond vacuum representation $\pi$ of 
$\A_c$ to this finite-dimensional subalgebra gives rise to an entire cyclic cocycle whose entire cyclic cohomology class is 
a complete noncommutative geometric invariant for the class of Ramond vacuum representations of the $N=2$ super-Virasoro net $\A_c$.
\end{proposition}

In the rational case $c<3$, this lets the Ramond vacuum representation theory of the $N=2$ super-Virasoro net $\A_c$ appear in still another form of noncommutative geometry, namely in terms of the entire cyclic cohomology of the finite-dimensional algebra 
$\AA_{c, {\Delta^c_{R\mathrm{vac}}}} \simeq \C^{\Delta^c_{R\mathrm{vac}}}$. This fact should also have interesting relations with the K-theoretical analysis of the representation theory of completely rational conformal nets given in \cite{CCH,CCHW}.

In the standard approach to $N=2$ superconformal field theory, the Ramond vacuum representations are related trough the spectral flow to the so called chiral (resp. antichiral) primary fields and the corresponding chiral ring, cf. \cite{LVW} and \cite[Sect. 5.3]{BP}. We now want to give a (partial) decription of these concepts in our operator algebraic framework. The idea is that the so called primary fields correspond to representations of $N=2$ superconformal nets and hence we have to define the concept of chiral (resp. antichiral) representation. As before we will restrict to the case of $N=2$ super-Virasoro nets.
\bigskip

\begin{definition}\label{def:ChR-ring}
We say that an irreducible Neveu-Schwarz representation $\pi$ of the $N=2$ super-Virasoro net $\A_c$ is a {\it chiral (resp. antichiral) 
representation} if for the unique (up to a phase) unit lowest energy vector $\Omega_\pi \in \H_\pi$ we have
\[
G^{+,\pi}_{-1/2}\Omega_\pi =0 \quad (\mathrm{resp.} \; G^{-,\pi}_{-1/2}\Omega_\pi =0).
\]
We denote the set of equivalence classes of chiral representations by $\Delta^c_{\mathrm{chir}}$.
\end{definition}

It easily follows from the $N=2$ NS super-Virasoro algebra (anti-) commutation relations  that a Neveu-Schwarz representation 
$\pi$ is chiral (antichiral) iff $h_\pi=q_\pi/2$ ($h_\pi=-q_\pi/2$, resp.), cf. e.g. \cite[Sect. 5.3]{BP}. In any case, the terminology is not to be confused with the usual notion of \emph{chiral symmetry} for fields where holomorphic and anti-holomorphic components decouple.

We now restrict ourselves to the rational case $c<3$. In this case the equivalence classes of chiral representations of $\A_c$ form an additive subgroup $\Z \Delta^c_{\mathrm{chir}} $ of the fusion ring generated by the equivalence classes of irreducible NS representations of $\A_c$ together with the fusion product in \eqref{eq:FermiFusions}. It is in general not invariant under the products of NS representations, but truncating,  i.e. putting equal to zero all the non chiral irreducible equivalence classes arising from the fusion product, the actual fermionic fusion rules in  \eqref{eq:FermiFusions}, which are represented by the fusion coefficients $N_{ij}^k\in\NN$, with $i,j,k$ running in the set of equivalence classes of irreducible Neveu-Schwarz representations of $\A_c$  adjusts this point. 
The corresponding truncated fusion coefficients are given by $\hat{N}_{ij}^k = N_{ij}^k $,  $ i, j ,k \in \Delta^c_{\mathrm{chir}}$. 
We denote by ``$*$'' the corresponding product operation on $\Z \Delta^c_{\mathrm{chir}}$.

\begin{proposition}\label{prop:ChR-ring}
Let $c<3$ be an allowed value for the central charge  in the $N=2$ discrete series representations. The product $*$ defined by the $\Delta^c_{\mathrm{chir}}$-truncated fusion rules on $\Z \Delta^c_{\mathrm{chir}}$ is commutative and associative. We call the corresponding ring the \emph{chiral ring} of the $N=2$ super-Virasoro net $\A_c$.
\end{proposition}

\begin{proof}
The equivalence classes of irreducible Neveu-Schwarz representations of $\A_c$ are labelled by $(l,m)$ like the corresponding irreducible unitary representations with central charge $c$ of the $N=2$ Neveu-Schwarz super-Virasoro algebra, cf. Section \ref{sec:coset}. The equivalence classes of chiral representations are then selected by the condition $2 h(l,m)= q(l,m)$. In light of Theorem \ref{th:SVir2-reps}(NS3), this means $m=-l$. Truncating then the fusion rules \eqref{eq:FermiFusions} according to this condition, at most one term remains in the sum \eqref{eq:FermiFusions}, namely
\[
 (l_1,-l_1)*(l_2,-l_2)=\left\lbrace
\begin{array}{l@{\; :\;}l}
(l_1+l_2,-l_1-l_2) & l_1+l_2\le n\\
0 & \textrm{otherwise}.
\end{array}
\right\rbrace = (l_2,-l_2)*(l_1,-l_1), 
\]
for all $l_1,l_2= 0,...,n$. In particular, this proves commutativity of $*$. Similarly, we check associativity:
\begin{align*}
 \Big((l_1,-l_1)*(l_2,-l_2)\Big) *(l_3,-l_3) =& \left\lbrace
\begin{array}{l@{\; :\;}l}
(l_1+l_2,-l_1-l_2)*(l_3,-l_3) & l_1+l_2\le n\\
0 & \textrm{otherwise}.
\end{array}\right.\\
=&\left\lbrace
\begin{array}{l@{\; :\;}l}
(l_1+l_2+l_3,-l_1-l_2-l_3) & l_1+l_2+l_3\le n\\
0 & \textrm{otherwise}.
\end{array}
\right.\\
=& (l_1,-l_1)*\Big( (l_2,-l_2)*(l_3,-l_3)\Big).
\end{align*}
\end{proof}

\begin{remark}\label{rem:ChR-original}
As mentioned in part above, in the literature the ``chiral ring" stands for the ring defined by the OPE at coinciding points of the so-called chiral primary fields, cf. \cite{LVW} and \cite[Sect. 5.3]{BP}. For the $N=2$ minimal model considered here this gives the same result that we obtained by means of the chiral representations and their truncated fusion rules, see the example at the end of Section 5.5 in 
\cite{BP}. This is of course not surprising if one recalls that in the formulation of the approaches to CFT based on pointlike localized fields the fusion rules are defined in terms of the OPE of primary fields. 

One can define in a similar way the rings corresponding to antichiral representations or antichiral primary fields. Note that the models we are considering are generated by fields depending on one light-ray coordinate only. In the general 2D case one has a richer structure of chiral/antichiral rings: $(c,c)$, $(a,a)$, $(c,a)$, $(a,c)$, \cf \cite[pp. 433, 437]{LVW}. 

Note also that from the OPE point of view the associativity of the ring product has a natural explanation, while it is not {\it a priori} evident from the point of view of truncated fusion rules. 
\end{remark}

We now come to the relation between the chiral representations and the Ramond vacuum representations of the $N=2$ super-Virasoro nets $\A_c$ with $c<3$.
Recall from Proposition  \ref{spflowrep} that composition of an irreducible Neveu-Schwarz representation $\pi$ of $\A_c$ with the spectral flow automorphism $\bar{\eta}_{-1/2}$ gives an irreducible Ramond representation  $\pi\circ \bar{\eta}_{-1/2}$. 
If $\pi$ denotes the irreducible representation of the $N=2$ Neveu-Schwarz super-Virasoro algebra corresponding to the
Neveu-Schwarz representation $\pi$ of $\A_c$ then the representation of the $N=2$ Ramond  super-Virasoro algebra corresponding to 
$\pi\circ \eta_{-1/2}$ is $\pi\circ \bar{\eta}_{-1/2}$.  The following argument is standard, see e.g. \cite[Sect. 5.4]{BP}. Let $\Omega_\pi$ be the unique (up to a phase) lowest energy unit vector for the irreducible Neveu-Schwarz representation $\pi$ of $\A_c$. Then, using the explicit expressions for the Ramond super-Virasoro algebra generators in the representation $\pi\circ \eta_{-1/2}$ in terms of the Neveu-Schwarz super-Virasoro algebra generators in the representation $\pi$, it is straightforward to see that $\Omega_\pi$ is a lowest energy vector also for the representation $\pi\circ \eta_{-1/2}$ which furthermore satisfies $G^{-,\pi \circ \eta_{-1/2}}_{0}\Omega_\pi = 0$. Moreover we have 
$G^{+,\pi \circ \eta_{-1/2}}_{0}\Omega_\pi = G^{+,\pi}_{-1/2}\Omega_\pi$ and it follows that $\pi$ is a chiral representation of $\A_c$ if and only if $\pi\circ \bar{\eta}_{-1/2}$ is a Ramond vacuum representation. Now, if $\pi$ is an irreducible Ramond representation of $\A_c$ then, again by Proposition  \ref{spflowrep},  $\pi\circ \bar{\eta}_{1/2}$ is an irreducible Neveu-Schwarz representation of $\A_c$. Hence, 
$\pi\circ \bar{\eta}_{1/2}$ is a chiral representation if and only if $\pi = \pi\circ \bar{\eta}_{1/2} \circ \bar{\eta}_{-1/2}$  is a Ramond vacuum representation of $\A_c$. Accordingly we have the following: 

\begin{proposition}\label{prop:ChR-spfl}
The spectral flow automorphism $\bar{\eta}_{1/2}$ of $\A_c$ gives rise to a one-to-one correspondence between 
$\Delta^c_{R\mathrm{vac}}$ and $\Delta^c_{\mathrm{chir}}$. As a consequence the noncommutative geometric invariants for 
$\Delta^c_{R\mathrm{vac}}$ in Proposition  \ref{prop:JLO-centre} give rise to noncommutative geometric invariants for the chiral ring 
$\Z \Delta^c_{\mathrm{chir}}$.
\end{proposition}

\vspace{2\baselineskip}
{\small

\noindent
{\bf Acknowledgements.} 
We would like to thank T.~Arakawa, V.~Dobrev, B.~Gato-Rivera, O.~Gray, A.~Jaffe, and V.~Kac for useful discussions and comments. Moreover, Y.K. and F.X. would like to thank the CMTP in Rome, R.H.,  R.L., and Y.K. the MIT Cambridge, and S.C. and R.L. Harvard University, Cambridge, for hospitable research stays during which parts of this work were accomplished. Most of this work was done while R.H. was at the Department of Mathematics of the University of Rome  ``Tor Vergata'' and subsequently at the School of Mathematics of Cardiff University as an INdAM-Cofund Marie Curie Fellow.
\vspace{2\baselineskip}

}

\end{document}